\documentclass{article}

\usepackage{
    amsfonts,
    amsmath,
    amssymb,
    amsthm,
    accents, 
    enumitem,
    float,
    graphicx,
    mathrsfs,
    shuffle,
    stmaryrd,
    subcaption,
    tikz,
    ytableau
}

\usepackage[margin=2.5cm]{geometry}

\setlist[enumerate]{label=\textup{(\arabic*)}}

\newtheorem{theorem}{Theorem}[section]
\newtheorem{proposition}[theorem]{Proposition}

\theoremstyle{definition}

\newtheorem{remark}[theorem]{Remark}

\usepackage[colorinlistoftodos]{todonotes}

\definecolor{melon}{rgb}{0.99,0.72,0.67}
\definecolor{cornflower}{rgb}{0.66,0.77,0.95}
\definecolor{lightgreen}{rgb}{0.67, 0.99, 0.67}
\definecolor{aqua}{rgb}{0, 1, 1}

\renewcommand{\operatorname}{\mathsf}
\renewcommand{\emptyset}{\varnothing}

\newcommand{\sym}{\mathfrak{S}}

\newcommand{\Sym}{\operatorname{Sym}}
\newcommand{\QSym}{\operatorname{QSym}}
\newcommand{\sSym}{\operatorname{sSym}}
\newcommand{\sQSym}{\operatorname{sQSym}}
\newcommand{\FQSym}{\operatorname{FQSym}}
\newcommand{\NCSym}{\operatorname{NCSym}}
\newcommand{\NCQSym}{\operatorname{NCQSym}}
\newcommand{\sNCSym}{\operatorname{sNCSym}}
\newcommand{\sNCQSym}{\operatorname{sNCQSym}}
\newcommand{\sFQSym}{\operatorname{sFQSym}}

\newcommand{\Asc}{\operatorname{Asc}}
\newcommand{\Cut}{\operatorname{Cut}}
\newcommand{\Des}{\operatorname{Des}}
\newcommand{\QSh}{\operatorname{QSh}}
\newcommand{\SSh}{\operatorname{SSh}}
\newcommand{\GDes}{\operatorname{GDes}}
\newcommand{\sComp}{\operatorname{sComp}}

\newcommand{\ind}{\operatorname{ind}}
\newcommand{\inv}{\operatorname{inv}}
\newcommand{\len}{\operatorname{len}}
\newcommand{\sgn}{\operatorname{sgn}}
\newcommand{\std}{\operatorname{std}}

\renewcommand{\min}{\operatorname{min}}
\renewcommand{\max}{\operatorname{max}}

\newcommand{\dbrack}[1]{[\kern-0.115em[#1]\kern-0.121em]}
\newcommand{\dangle}[1]{\langle\kern-0.18em\langle #1\rangle\kern-0.18em\rangle}
\newcommand{\ubar}[1]{\underaccent{\bar}{#1}}
\def\blue{\color{blue}}\def\red{\color{red}}

\newcommand{\id}{\operatorname{id}}

\newcommand{\Q}{\mathbb{Q}}
\newcommand{\N}{\mathbb{N}}
\newcommand{\Z}{\mathbb{Z}}

\newcommand{\M}{\mathcal{M}}

\title{On quasisymmetric functions in superspace}

\author{
    Diego Arcis\footnote{Departamento de Matem\'aticas, Universidad de La Serena, Cisternas 1200 -- 1700000 La Serena, Chile (\texttt{diego.arcis@userena.cl}).}
\and
    Camilo Gonz\'alez\footnote{Departamento de Matem\'atica, Universidad de Concepci\'on, Casilla 160-C -- 4030000 Concepci\'on, Chile (\texttt{camgonzalezp@udec.cl}).}
\and
    Sebasti\'an M\'arquez\footnote{Departamento de Matem\'aticas, Universidad Aut\'onoma de Chile, Pedro de Valdivia 425 -- 7500000 Providencia, Chile (\texttt{sebastian.marquez@uautonoma.cl}).}
}

\date{}

\usepackage{hyperref}

\begin{document}

\maketitle

\begin{abstract}
Quasisymmetric functions in superspace were introduced as a natural extension of classical quasisymmetric functions involving both commuting and anticommuting variables. In this paper, we first provide a characterization of the algebra of quasisymmetric functions in superspace as an algebra of invariants under a quasisymmetrizing action of the symmetric group. Furthermore, we complete the superspace analogue of the classical hierarchy of combinatorial Hopf algebras by introducing the algebra of quasisymmetric functions in noncommuting variables in superspace. We endow this algebra with a Hopf superalgebra structure and thoroughly investigate its $Q$-basis and monomial basis, which are indexed by set supercompositions. By restricting to the minimal elements of the underlying poset, we construct the Hopf superalgebra of superpermutations, serving as the superspace analogue of the Malvenuto--Reutenauer algebra. We provide explicit product and coproduct formulas for these bases in terms of super-shuffles and global descents. Finally, via an abelianization morphism, we apply these noncommutative structures to derive a product formula for fundamental quasisymmetric functions in superspace.
\end{abstract}

\tableofcontents

\section{Introduction}

The interplay between symmetric functions, their generalizations, and combinatorial Hopf algebras has been a central theme in algebraic combinatorics~\cite{Mac99,GrRe20,AgBeSo06}. Classical symmetric functions and quasisymmetric functions are fundamental in the study of symmetric groups and representation theory~\cite{Mac99,Sa01}. Their noncommutative analogues, the algebras of symmetric functions in noncommuting variables and noncommutative quasisymmetric functions, have further deepened this connection, revealing rich underlying partial orders and basis transformations~\cite{Wol36,RoSa06,BeZa09}. Central to this hierarchy lies the Malvenuto--Reutenauer Hopf algebra of permutations~\cite{MaRe95,AgSo02}, which serves as a terminal object governing the product and coproduct rules of these structures through permutations and the weak Bruhat order.

In recent years, the classical framework has been extended to superspace, motivated by theories in physics involving both commuting and anticommuting variables~\cite{DeLaMa06}. Following the development of Jack polynomials in superspace~\cite{DeLaMa03,DeLaMa04}, the algebra of symmetric functions in superspace $\sSym$ was rigorously formalized and has since revealed a remarkably rich combinatorial structure~\cite{DeLaMa06}. Many fundamental features of the classical theory admit natural extensions to this setting, including analogues of Schur functions and Macdonald polynomials, together with their associated combinatorics such as Pieri rules and triangularity properties~\cite{BlDeLaMa12,JoLa17,GaJoLa19}. In parallel, important developments have emerged in related directions, notably the study of coinvariant rings in superspace, which has been actively investigated in recent years~\cite{RhoWi24,AnCoKaMuRho25}.

Subsequently, Fishel, Lapointe, and Pinto~\cite{FiLaPi19} introduced the algebra of quasisymmetric functions in superspace $\sQSym$, uncovering a combinatorial framework governed by dotted compositions and a superspace analogue of the fundamental basis. Furthermore, from a categorical perspective, it has been shown that $\sQSym$ plays a universal role as a terminal object in the category of combinatorial Hopf superalgebras~\cite{HaYaYa25}. Recently, the symmetric functions in noncommuting variables in superspace $\sNCSym$ were also introduced~\cite{ArGoMa25}. However, the noncommutative quasisymmetric counterpart, as well as the overarching superspace analogue of the Malvenuto--Reutenauer algebra, remained unexplored.

In this paper, we fill this gap by completing the superspace hierarchy of combinatorial Hopf algebras. We formally introduce the algebra of quasisymmetric functions in noncommuting variables in superspace $\sNCQSym$ and construct the Hopf superalgebra of free quasisymmetric functions in superspace $\sFQSym$. Our approach closely mirrors the robust algebraic framework of classical $\FQSym$, generalizing it to accommodate the parity conditions imposed by fermionic variables.\medskip

We summarize here the main results of the paper.

\medskip

Let $x=(x_1,x_2,\ldots)$ and $\theta=(\theta_1,\theta_2,\ldots)$ be two sets of variables satisfying the superspace commutativity relations: $x_ix_j=x_jx_i$, $x_i\theta_j=\theta_jx_i$, and $\theta_i\theta_j=-\theta_j\theta_i$ for all $i,j$. A formal power series $f(x,\theta)$ is defined as \emph{quasisymmetric in superspace} if any two monomials appearing in $f$ with the same relative ordering in both sets of variables simultaneously have identical coefficients (Subsection~\ref{003}). The algebra of such functions, denoted by $\sQSym$, is indexed by dotted compositions. We extend Hivert's quasisymmetrizing action to the superspace setting by introducing \emph{pseudo-dotted compositions}. This framework allows for a description of simple transpositions that respects the relative ordering of variables, leading to the following invariant-theoretic characterization.

\begin{proof}[{\bf Proposition~\ref{031}}]\emph{
A formal power series $f\in\Q^\theta\dbrack{x}$ belongs to $\sQSym$ if and only if it is invariant under the quasisymmetrizing action in superspace, that is, $\sigma f=f$ for all $\sigma\in\sym_\infty$.
}\end{proof}

We then introduce $\sNCQSym$, the algebra of quasisymmetric functions in noncommuting variables in superspace. This construction generalizes the notion of packed words and characterizes quasisymmetry through superspace standardization. Elements of $\sNCQSym$ are naturally indexed by set supercompositions.

\begin{proof}[{\bf Theorem~\ref{032}}]\emph{
The algebra $\sNCQSym$ carries a natural structure of a graded Hopf superalgebra, with product and coproduct rules governed by quasi-shuffle operations in superspace.
}\end{proof}

The combinatorial structure of $\sNCQSym$ is governed by set supercompositions, which generalize the role of set compositions in the classical theory. Within this framework, we identify a distinguished class of minimal elements under a natural partial order, which we term \emph{superpermutations}. These objects generate a superspace analogue of the Malvenuto--Reutenauer algebra.

\begin{proof}[{\bf Proposition~\ref{033}}]\emph{
The subspace $\sFQSym\subset\sNCQSym$ spanned by the functions $\{Q_I\}$, where $I$ ranges over superpermutations, is a sub-Hopf superalgebra. This structure serves as the superspace analogue of the algebra of free quasisymmetric functions.
}\end{proof}

Finally, we establish a fundamental connection between the noncommutative and commutative settings via the abelianization morphism $\pi:\sNCQSym\to\sQSym$.

\begin{proof}[{\bf Theorem~\ref{002}}]\emph{
The natural abelianization morphism maps the $\{Q_I\}$ basis of $\sFQSym$ onto the fundamental basis of $\sQSym$. Consequently, the super-shuffle product in the noncommutative setting projects directly onto the product formula for fundamental quasisymmetric functions in superspace.
}\end{proof}

\medskip

The paper is organized as follows. In Section~\ref{021}, we review the preliminary notions of Hopf superalgebras and our conventions regarding superspace alphabets.

In Section~\ref{022}, we revisit the algebra of quasisymmetric functions in superspace $\sQSym$ (Subsection~\ref{003}). Extending the classical work of Hivert~\cite{Hi00}, we define a quasisymmetrizing action of the finitary symmetric group $\sym_\infty$ on superspace monomials. We prove that $\sQSym $ is precisely the algebra of invariants under this action, tightly linking the commutative superspace variables with group actions (Subsection~\ref{034}). We also recall the basis of fundamental quasisymmetric functions in superspace $L_\alpha$ (Subsection~\ref{035}).

In Section~\ref{023}, we introduce the main object of study: the algebra $\sNCQSym$ (Subsection~\ref{036}). We show that $\sNCQSym$ naturally inherits a Hopf superalgebra structure (Subsection~\ref{038}). Similar to $\sQSym$, we demonstrate that $\sNCQSym$ can be characterized as an invariant algebra under a noncommutative analogue of the quasisymmetrizing action (Subsection~\ref{037}). Furthermore, we formally position $\sNCSym$ as a sub-Hopf superalgebra within this new space (Subsection~\ref{039}).

Section~\ref{024} constitutes the core combinatorial machinery of the paper. We introduce a new basis for $\sNCQSym$, the $Q$-basis, indexed by set supercompositions. By defining a suitable partial order on set supercompositions that generalizes the classical refinement order (Subsection~\ref{005}), we study the product and coproduct rules for this basis, which naturally involve the combinatorics of super-shuffles (Subsection~\ref{040}). Restricting our attention to the minimal elements of this poset, which we term superpermutations, we define $\sFQSym$ as the sub-Hopf superalgebra generated by these elements. We define the monomial basis $\M$ for $\sFQSym$ via M\"{o}bius inversion on the super left weak order, and provide formulas for its product and coproduct utilizing global descents (Subsection~\ref{041}). Finally, we exploit the natural abelianization morphism from noncommuting to commuting variables to bridge our new noncommutative structures with the existing commutative ones. We show how the super-shuffle product of the $Q$-basis perfectly projects onto the product of fundamental quasisymmetric functions $L_\alpha$, providing a transparent combinatorial formula for their multiplication (Subsection~\ref{025}).

\section{Preliminaries}\label{021}

\subsection{Notations and conventions}

Throughout this paper, the word ``algebra'' refers to an associative unitary algebra over the field of rational numbers $\Q$. We denote the set of positive integers by $\N=\{1,2,\ldots\}$ and the set of nonnegative integers by $\N_0=\N\cup\{0\}$. For $n\in\N$, we define the interval $[n]=\{1,\ldots,n\}$ and $[n]_0=[n]\cup\{0\}$. If $a$ is a finite sequence, its length is denoted by $\ell(a)$. We also define the set of dotted nonnegative integers as $\dot{\N}_0=\{\dot{a}\mid a\in\N_0\}$.

Given a sequence $I=(I_1,\ldots,I_k)$ of subsets of $\N_0$, its \emph{standardization} $\std(I)$ is the sequence obtained from $I$ by replacing its nonzero elements with those of $[n]$ via the unique order-preserving bijection $(I_1\cup\cdots\cup I_k)\setminus\{0\}\to[n]$, where $n=|(I_1\cup\cdots\cup I_k)\setminus\{0\}|$.

\subsection{Hopf superalgebras}

For an algebra $A$, the \emph{product} of $A$ is the linear map $m:A\otimes A\to A$ defined by $m(a\otimes b)=ab$, and the \emph{unit} of $A$ is the linear map $\iota:\Q\to A$ defined by $\iota(k)=k1$. These maps satisfy the associativity condition $m(m\otimes\id)=m(\id\otimes m)$ and the unit axioms $m(\id\otimes\iota)=\id=m(\iota\otimes\id)$.

A \emph{superalgebra} is an algebra $A$ equipped with a $\Z_2$-grading; that is, $A$ admits a direct sum decomposition $A=A_0\oplus A_1$ such that $A_iA_j\subseteq A_{i+j}$ for all $i,j\in\Z_2$. Elements of $A_0$ are called \emph{even}, and those of $A_1$ are called \emph{odd}. Specifically, an element $a\in A_i$ is said to be \emph{homogeneous} of \emph{parity} $|a|:=i$.

Given superalgebras $A$ and $B$, the tensor product space $A\otimes B$ is naturally equipped with a superalgebra structure, called the \emph{super tensor product} of $A$ with $B$, with multiplication defined on homogeneous elements by $(a\otimes b)\cdot(c\otimes d)=(-1)^{|b||c|}ac\otimes bd$, and identity $1\otimes 1$. The $\Z_2$-grading is given by $(A\otimes B)_0=(A_0\otimes B_0)\oplus(A_1\otimes B_1)$ and $(A\otimes B)_0=(A_0\otimes B_1)\oplus(A_1\otimes B_0)$.

A map $\psi:A\to B$ between superalgebras is called \emph{even} (resp. \emph{odd}) if $\psi(A_i)\subseteq B_i$ (resp. $\psi(A_i)\subseteq B_{i+1}$) for each $i\in\Z_2$. In particular, $\Q$ can be regarded as a superalgebra with the trivial $\Z_2$-grading $\Q_0=\Q$ and $\Q_1=\{0\}$. Consequently, for any superalgebra $A$, the product $m:A\otimes A\to A$ and the unit $\iota:\Q\to A$ are even linear maps, where $A\otimes A$ is endowed with the super tensor product superalgebra structure.

A \emph{superbialgebra} is a superalgebra $H$ together with even linear maps $\Delta:H\to H\otimes H$ and $\varepsilon:H\to\Q$, called \emph{coproduct} and \emph{counit}, respectively, such that the following conditions hold:\begin{gather*}
(\id\otimes\Delta)\Delta=(\Delta\otimes\id)\Delta,\qquad(\varepsilon\otimes\id)\Delta=\id=(\id\otimes\varepsilon)\Delta,\\
\Delta(1)=1,\qquad\Delta(ab)=\Delta(a)\Delta(b),\qquad\varepsilon(1)=1,\quad\varepsilon(ab)=\varepsilon(a)\varepsilon(b).
\end{gather*}

A \emph{Hopf superalgebra} is a superbialgebra $H$ endowed with an even linear map $S:H\to H$, called the \emph{antipode}, such that $m(\id\otimes S)\Delta=\iota\varepsilon=m(S\otimes\id)$.

\section{Quasisymmetric functions in superspace}\label{022}

In this section, we study the algebra of quasisymmetric functions in superspace, a framework introduced in \cite[Section 5]{FiLaPi19} that extends the classical theory of quasisymmetric functions by incorporating anticommuting variables. We first recall its formal definition and its associated monomial basis \cite[Subsection 5.1]{FiLaPi19}. Subsequently, we demonstrate that, much like $\sSym$, the algebra $\sQSym$ can also be realized as an algebra of invariants under an action of the finitary symmetric group $\sym_\infty$. We call this the \emph{quasisymmetrizing action in superspace}, which naturally extends the classical quasisymmetrizing action on $\QSym$ \cite[Section 3]{Hi00}. Finally, we recall the basis of fundamental quasisymmetric functions in superspace~\cite[Subsection 5.5]{FiLaPi19} \cite{FiGaLaPi25}.

\subsection{The algebra of quasisymmetric functions in superspace}\label{003}

Let $\Q^\theta\dbrack{x}$ be the algebra of formal power series of bounded degree in the variables $x=(x_1,x_2,\ldots)$ and $\theta=(\theta_1,\theta_2,\ldots)$, subject to the relations $x_ix_j=x_jx_i$, $x_i\theta_j=\theta_jx_i$, and $\theta_i\theta_j=-\theta_j\theta_i$. Note that $\theta_i^2= 0$ for all $i$, and that every monomial $u\in\Q^\theta\dbrack{x}$ can be uniquely expressed in the normal form $u=q\theta_{i_1}\cdots\theta_{i_m}x_{j_1}\cdots x_{j_n}$, where $q\in\Q$, $i_1<\cdots<i_m$, and $j_1\leq\cdots\leq j_n$. In what follows, we shall always assume that monomials are written in this form. For a monomial $u=\theta_{i_1}\cdots\theta_{i_m}x_{j_1}\cdots x_{j_n}$, the \emph{set of indices} of $u$ is defined as $\ind(u)=\{i_1,\ldots,i_m,j_1,\ldots,j_n\}$.

A formal power series $f=f(x,\theta)\in\Q^\theta\dbrack{x}$ is called a \emph{symmetric function in superspace} if it is invariant under any simultaneous permutation of the indices of both $x$ and $\theta$; that is, $f(x,\theta)=f(\sigma x,\sigma\theta)$ for all $\sigma\in\sym_n$ and all $n\in\N$, where $\sigma x=(x_{\sigma(1)},x_{\sigma(2)},\ldots)$ and $\sigma\theta=(\theta_{\sigma(1)},\theta_{\sigma(2)},\ldots)$. The set of all such symmetric functions forms a subalgebra of $\Q^\theta\dbrack{x}$ called the \emph{algebra of symmetric functions in superspace}, denoted by $\sSym$ \cite{DeLaMa03,DeLaMa04}. This algebra naturally contains the classical algebra of symmetric functions $\Sym$ as the subalgebra of elements that are independent of the variables $\theta$. By definition, both $\sSym$ and its subalgebra $\Sym$ are algebras of invariants under the natural action of the finitary symmetric group $\sym_\infty$ on the sets of variables $x$ and $\theta$.

A formal power series $f\in\Q^\theta\dbrack{x}$ is called a \emph{quasisymmetric function in superspace} if, for any sequence of nonnegative integers $a_1,\ldots,a_k$ and any boolean sequence $\epsilon_1,\dots,\epsilon_k\in\{0,1\}$ such that $a_i+\epsilon_i\geq 1$ for all $i\in[k]$, the coefficient of the monomial $\theta_{i_1}^{\epsilon_1}\cdots\theta_{i_k}^{\epsilon_k}x_{i_1}^{a_1}\cdots x_{i_k}^{a_k}$ occurring in $f$ is the same for all strictly increasing sequences of indices $i_1<\cdots<i_k$ \cite[Definition 5.2]{FiLaPi19}. The set of all such functions forms a subalgebra of $\Q^\theta\dbrack{x}$ called the \emph{algebra of quasisymmetric functions in superspace}, denoted by $\sQSym$ \cite[Proposition 5.5]{FiLaPi19}. This algebra naturally contains $\sSym$ as a subalgebra and contains the classical algebra of quasisymmetric functions $\QSym$ as the subalgebra of elements that are independent of the variables $\theta$.

The natural basis for $\sQSym$ is indexed by dotted compositions. A \emph{dotted composition} is a finite sequence $\alpha=(\alpha_1,\ldots,\alpha_k)$, where each component $\alpha_i\in\N\cup\dot{\N}_0$ \cite[Definition 5.1]{FiLaPi19}. Note that classical integer compositions are naturally recovered as dotted compositions with no dotted components. The \emph{monomial quasisymmetric function in superspace} indexed by $\alpha$ is defined as the formal power series \cite[Definition 5.3]{FiLaPi19}:\[M_\alpha=\sum_{i_1<\cdots<i_k}\theta_{i_1}^{\bar{\alpha}_1}\cdots\theta_{i_k}^{\bar{\alpha}_k}x_{i_1}^{\ubar{\alpha}_1}\cdots x_{i_k}^{\ubar{\alpha}_k},\]where, for each $i\in[k]$, $\bar{\alpha}_i=0$ and $\ubar{\alpha}_i=\alpha_i$ if $\alpha_i\in\N$, and $\bar{\alpha}_i=1$ and $\ubar{\alpha}_i=a$ if $\alpha_i=\dot{a}$ for some $a\in\N_0$. The set of all $M_\alpha$, where $\alpha$ ranges over all dotted compositions, forms a linear basis for $\sQSym$.

\subsection{Quasisymmetrizing action in superspace}\label{034}

As shown in \cite[Proposition 3.15]{Hi00}, the classical algebra $\QSym$ can be realized as an algebra of invariants under an action of the finitary symmetric group, known as the \emph{quasisymmetrizing action} \cite[Section 3]{Hi00}. In this subsection, we extend this action to the algebra of quasisymmetric functions in superspace, and show that $\sQSym$ is similarly an algebra of invariants under this extended action of $\sym_\infty$. These results suggest that the algebra $\sQSym$ is one of the invariant algebras belonging to the framework established in \cite{PreArGoMa26B}.

We first observe that for every monic monomial $u\in\Q^\theta\dbrack{x}$, there exists a minimal nonnegative integer $n$ such that $u=\theta_1^{\epsilon_1}\cdots\theta_n^{\epsilon_n}x_1^{a_1}\cdots x_n^{a_n}$, where $a_1,\ldots,a_n\in\N_0$, $\epsilon_1,\ldots,\epsilon_n\in\{0,1\}$, and $a_n+\epsilon_n\geq1$. This observation allows us to characterize the underlying variables of these monomials by means of dotted pseudo-compositions. A \emph{dotted pseudo-composition} is a finite sequence $\alpha=[\alpha_1,\ldots,\alpha_k]$, where each component $\alpha_i\in\N_0\cup\dot{\N}_0$. The dotted pseudo-composition associated with the monomial $u$ is defined as $\alpha(u)=[\alpha_1,\ldots,\alpha_n]$, where $\alpha_i=a_i$ if $\epsilon_i=0$, and $\alpha_i=\dot{a}_i$ if $\epsilon_i=1$. Conversely, given a dotted pseudo-composition $\alpha=[\alpha_1,\ldots,\alpha_n]$, let $\ubar{\alpha}_i\in\N_0$ denote its underlying integer value and $\bar{\alpha}_i\in\{0,1\}$ denote its dot indicator. Then, the corresponding monic monomial is given by $u_\alpha=\theta_1^{\bar{\alpha}_1}\cdots\theta_n^{\bar{\alpha}_n}x_1^{\ubar{\alpha}_1}\cdots x_n^{\ubar{\alpha}_n}$. For instance, if $u=\theta_2\theta_5x_3^3x_4x_5^2x_7$, its dotted pseudo-composition is $\alpha(u)=[0,\dot{0},3,1,\dot{2},0,1]$.

These assignments establish a canonical bijection between the set of monic monomials and the set of all dotted pseudo-compositions. Using this bijection, we can define an action of the finitary symmetric group $\sym_\infty$ on the monic monomials, which then extends naturally to the entire algebra $\Q^\theta\dbrack{x}$. Let $s_i$ be the simple transpositions exchanging $i$ with $i+1$. We define the \emph{quasisymmetrizing action} of $s_i$ on a dotted pseudo-composition $\alpha=[\alpha_1,\ldots,\alpha_k]$ by\[s_i\alpha=\begin{cases}
\,[\alpha_1,\ldots,\alpha_{i-1},\alpha_{i+1},\alpha_i,\alpha_{i+2},\ldots,\alpha_k]&\text{if }i<k\text{ and }0\in\{\alpha_i,\alpha_{i+1}\}\\
\,[\alpha_1,\ldots,\alpha_{k-1},0,\alpha_k]&\text{if }i=k\\
\,[\alpha_1,\ldots,\alpha_k]&\text{otherwise}.
\end{cases}\]It is straightforward to verify that the operators $s_i$ satisfy the Coxeter relations, ensuring that this indeed defines an action of $\sym_\infty$ on $\Q^\theta\dbrack{x}$. For instance, if $u=\theta_2\theta_5x_3^3x_4x_5^2x_7$ and $\sigma=s_5s_3s_2$, we have\[\sigma u=\sigma\,\theta_1^0\theta_2^{{\red1}}\theta_3^{{\red0}}\theta_4^{{\red0}}\theta_{{\blue5}}^1\theta_6^0\theta_7^0x_1^0x_2^{{\red0}}x_3^{{\red3}}x_4^{{\red1}}x_{{\blue5}}^2x_6^0x_7^1=s_5\theta_1^0\theta_2^1\theta_3^0\theta_4^0\theta_{{\blue5}}^{{\red1}}\theta_6^{{\red0}}\theta_7^0x_1^0x_2^0x_3^3x_4^1x_{{\blue5}}^{{\red2}}x_6^{{\red0}}x_7^1=\theta_2\theta_{{\blue6}}x_3^3x_4x_{{\blue6}}^2x_7.\]

For a dotted pseudo-composition $\beta$, we denote by $\hat{\beta}$ the unique dotted composition obtained by removing all its parts equal to $0$. If $u$ is the unique monic monomial defined by $\beta$, we call $\hat{\beta}$ the \emph{dotted composition associated with $u$}. Given a dotted composition $\alpha$ and a finite subset $A\subset\N$ with $|A|=\ell(\alpha)$, we denote by $A^\alpha$ the unique monic monomial whose set of indices is $A$ and whose associated dotted composition is $\alpha$. For instance, the monomial $u=\theta_2\theta_5x_3^3x_4x_5^2x_7$ is defined by the dotted pseudo-composition $[{\red0},\dot{0},3,1,\dot{2},{\red0},1]$, and so $u=\{2,3,4,5,7\}^{(\dot{0},3,1,\dot{2},1)}$.

Since the quasisymmetrizing action only exchanges adjacent components when at least one is zero, it strictly preserves the relative order of the nonzero parts. Thus, the associated dotted composition remains invariant, and the action merely permutes the underlying set of indices. This yields the following description.
\begin{proposition}\label{016}
Let $\alpha$ be a dotted composition, $A \subset \N$ be a finite subset with $|A|=\ell(\alpha)$, and $\sigma\in\sym_\infty$. Then the quasi-symmetrizing action on the monomial $A^\alpha$ is given by $\sigma\cdot A^\alpha=\sigma(A)^\alpha$, where $\sigma(A)=\{\sigma(a)\mid a\in A\}$.
\end{proposition}
\begin{proof}
It suffices to prove the statement for a simple transposition $\sigma=s_i$. Let $u=A^\alpha$ and let $\beta=[\beta_1,\ldots,\beta_k]$ be the unique dotted pseudo-composition defining $u$, so that $\hat{\beta}=\alpha$ and $\ind(u) = A$. By definition, the quasi-symmetrizing action of $\sigma$ exchanges $\beta_i$ and $\beta_{i+1}$ only if at least one of them is zero, where we conventionally set $\beta_{i+1}=0$ if $i\geq k$. Consequently, $\widehat{\sigma\beta}=\hat{\beta}=\alpha$, meaning $\alpha$ is the dotted composition associated with $\sigma u$. On the other hand, since the action simply permutes the components of $\beta$ according to $\sigma$, the positions of the nonzero parts are shifted accordingly. Thus, the new set of indices is exactly $\sigma(A)$, yielding $\ind(\sigma u)=\sigma(A)$. Therefore, $\sigma\cdot u=\sigma(A)^\alpha$.
\end{proof}

For instance, if $u=\theta_2\theta_{{\blue5}}x_3^3x_4x_{{\blue5}}^2x_7$ and $\sigma=s_5s_3s_2$, we have\[\sigma u=\sigma(\{2,3,4,{\blue5},7\})^{(\dot{0},3,1,\dot{\red2},1)}=\{4,2,3,{\blue6},7\}^{(\dot{0},3,1,\dot{\red2},1)}=\{2,3,4,{\blue6},7\}^{(\dot{0},3,1,\dot{\red2},1)}=\theta_2\theta_{{\blue6}}x_3^3x_4x_{{\blue6}}^2x_7.\]

Finally, we show that the quasisymmetrizing action provides a complete characterization of the algebra quasisymmetric functions in superspace.
\begin{proposition}\label{031}
A formal power series $f$ of bounded degree in $\Q^\theta\dbrack{x}$ is a quasisymmetric function in superspace if and only if $\sigma f=f$ for all permutations $\sigma\in\sym_\infty$. 
\end{proposition}
\begin{proof}
Given $f\in Q^\theta\dbrack{x}$, we can uniquely expand $f$ as\[f=\sum_\alpha\!\sum_{\substack{A\subset\N\\\ |A|=\ell(\alpha)}}C_{A,\alpha}A^\alpha,\]where the outer sum runs over all dotted compositions $\alpha$ and $c_{A,\alpha}\in\Q$. By Proposition~\ref{016}, the quasisymmetrizing action preserves the dotted composition and merely permutes the underlying set of indices.

Suppose first that $f$ is a quasisymmetric function in superspace. As shown in Subsection~\ref{003}, $f$ can be expressed as a linear combination of the monomial quasisymmetric functions in superspace $M_\alpha=\sum_{|A|=\ell(\alpha)}A^\alpha$. This means the coefficient $C_{A,\alpha}$ depends only on $\alpha$ and not on $A$. Since $\sigma$ acts as a bijection on the collection of all finite subsets of $\N$ of size $\ell(\alpha)$, it simply permutes the terms in the inner sum, leaving $f$ completely invariant. Hence, $\sigma f=f$ for all permutation $\sigma\in\sym_\infty$.

Conversely, assume that $\sigma f=f$ for all $\sigma\in\sym_\infty$. Fix a dotted composition $\alpha$ and choose any two finite subsets $A,B\subset\N$ of size $\ell(\alpha)$. Since both sets have the same finite cardinality, there exists a permutation $\sigma$ such that $\sigma(A)=B$, which implies $\sigma A^\alpha=B^\alpha$. Because $f$ is invariant under $\sigma$, the coefficients of $A^\alpha$ and $B^\alpha$ in the expansion of $f$ must coincide. Therefore, the coefficient $c_{A,\alpha}$ depends exclusively on $\alpha$, meaning $f$ is a linear combination of the elements $M_\alpha$. Thus, $f$ is a quasisymmetric function in superspace.
\end{proof}

\subsection{Fundamental symmetric functions in superspace}\label{035}

Given two dotted compositions $\alpha$ and $\beta$, we say that $\beta$ \emph{covers} $\alpha=(\alpha_1,\ldots,\alpha_k)$ if there exists an index $i\in[k-1]$ such that neither $\alpha_i$ nor $\alpha_{i+1}$ is dotted, and $\beta=(\alpha_1,\dots,\alpha_{i-1},\alpha_i+\alpha_{i+1},\alpha_{i+2},\ldots,\alpha_k)$. The set of all dotted compositions is partially ordered by the reflexive and transitive closure of this covering relation, denoted by $\preceq$ \cite[Subsection 5.3]{FiLaPi19}. See Figure~\ref{000}. Observe that when restricted to classical integer compositions, this relation coincides exactly with the standard refinement order.

\begin{figure}[H]\centering
\begin{tikzpicture}[node distance=2cm]
\node(a)at(+0,3){$_{(1,{\red2},\dot{0},1,\dot{3},{\blue3})}$};
\node(b)at(-2.4,2){$_{(1,{\red2},\dot{0},1,\dot{3},{\blue2},{\blue1})}$};
\node(c)at(+0,2){$_{(1,{\red2},\dot{0},1,\dot{3},{\blue1},{\blue2})}$};
\node(d)at(+2.4,2){$_{(1,{\red1},{\red1},1,\dot{0},\dot{3},{\blue3})}$};
\node(e)at(-2.4,1){$_{(1,{\red2},\dot{0},1,\dot{3},{\blue1},{\blue1},{\blue1})}$};
\node(f)at(+0,1){$_{(1,{\red1},{\red1},\dot{0},1,\dot{3},{\blue2},{\blue1})}$};
\node(g)at(+2.4,1){$_{(1,{\red1},{\red1},\dot{0},1,\dot{3},{\blue1},{\blue2})}$};
\node(h)at(+0,0){$_{(1,{\red1},{\red1},\dot{0},1,\dot{3},{\blue1},{\blue1},{\blue1})}$};
\draw(a)--(b);\draw(a)--(c);\draw(a)--(d);\draw(b)--(f);\draw(c)--(g);\draw(d)--(g);\draw(f)--(h);\draw(g)--(h);\draw(b)--(e);\draw(c)--(e);\draw(c)--(e);\draw(d)--(f);\draw(e)--(h);
\end{tikzpicture}
\caption{Elements smaller than or equal to $(1,2,\dot{0},1,\dot{3},3)$.}\label{000}
\end{figure}
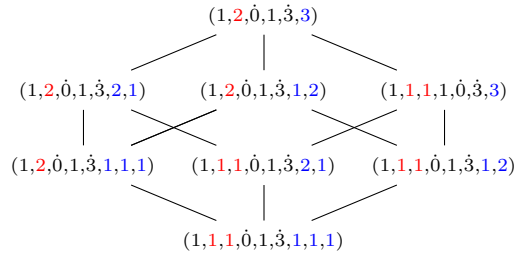
For a dotted composition $\alpha$, the \emph{fundamental quasisymmetric function in superspace} $L_\alpha$ is defined as the formal power series \cite[Definition 5.12, Equation (5.17)]{FiLaPi19}:\[L_\alpha=\sum_{\beta\preceq\alpha}M_\beta.\]For instance, by using Figure~\ref{000}, we obtain:\[\begin{array}{rcl}
L_{(1,2,\dot{0},1,\dot{3},3)}&=&M_{(1,2,\dot{0},1,\dot{3},3)}+M_{(1,2,\dot{0},1,\dot{3},2,1)}+M_{(1,2,\dot{0},1,\dot{3},1,2)}+M_{(1,1,1,\dot{0},1,\dot{3},3)}\,\,+\\[0.1cm]&&M_{(1,2,\dot{0},1,\dot{3},1,1,1)}+M_{(1,1,1,\dot{0},1,\dot{3},2,1)}+M_{(1,1,1,\dot{0},1,\dot{3},1,2)}+M_{(1,1,1,\dot{0},1,\dot{3},1,1,1)}.
\end{array}\]

We remark that there exists another family of fundamental quasisymmetric functions in superspace, which has been less studied in the literature. These alternative functions are defined with respect to a different partial order on dotted compositions, one that explicitly permits the merging of dotted components with non-dotted ones \cite[Definition 5.12, Equation (5.18)]{FiLaPi19}.

This subsection concludes with a characterization of the structure of the lower intervals in the poset of dotted compositions, which will be useful in Subsection~\ref{005}.

Recall that a \emph{Boolean lattice} of rank $n$ is a poset isomorphic to the lattice $B_n$ of subsets of $[n]$ ordered by inclusion \cite[Definition 11.13]{Shah22}.
\begin{proposition}\label{006}
Let $\alpha=(\alpha_1,\ldots,\alpha_k)$ be a dotted composition, and let $\alpha_{i_1},\ldots,\alpha_{i_s}$ be the non-dotted components of $\alpha$. Then, the interval $\alpha^{\,\downarrow}:=\{\beta\mid\beta\preceq\alpha\}$ is isomorphic to the product of Boolean lattices $B_{\alpha_{i_1}-1}\times\cdots\times B_{\alpha_{i_s}-1}$. In particular, it is a Boolean lattice of rank $(\alpha_{i_1}-1)+\cdots+(\alpha_{i_s}-1)$.
\end{proposition}
\begin{proof}
By definition, $\beta\preceq\alpha$ if and only if $\beta$ can be obtained by concatenating dotted compositions $\beta_1,\ldots,\beta_k$, where $\beta_{i_j}\preceq(\alpha_{i_j})$ for all $j\in[s]$, and $\beta_i=(\alpha_i)$ whenever $\alpha_i$ is dotted. It is well known that, for a positive integer $n$, the poset of compositions smaller than or equal to $(n)$ is isomorphic to the Boolean lattice $B_{n-1}$ of subsets of $[n-1]$ \cite[Section~1.2]{St97}. Indeed, if $a=(a_1,\ldots,a_r)\preceq(n)$, its classical descent set is $D(a)=\{a_1,a_1+a_2,\ldots,a_1+\cdots+a_{r-1}\}$ and the isomorphism is $a\mapsto[n-1]\setminus D(a)$. Hence, as the refinements $\beta_{i_1},\ldots,\beta_{i_s}$ occur independently, the map $\beta\mapsto([\alpha_{i_1}-1]\setminus D(\beta_{i_1}),\ldots,[\alpha_{i_s}-1]\setminus D(\beta_{i_s}))$ defines the desired poset isomorphism between $\alpha^{\,\downarrow}$ and $B_{\alpha_{i_1}-1}\times\cdots\times B_{\alpha_{i_s}-1}$.
\end{proof}

\section{Quasisymmetric functions in noncommuting variables in superspace}\label{023}

Classical symmetric and quasisymmetric functions have been naturally extended to the setting of noncommuting variables, yielding the algebras $\NCSym$ \cite{Wol36,RoSa06} and $\NCQSym$ \cite[Section 5]{BeZa09}~\cite[Subsection 2.2]{NoThWi10}. Recently, this framework was further extended to superspace with the introduction of $\sNCSym$, the \emph{algebra of symmetric functions in noncommuting variables in superspace} \cite{ArGoMa25}. In this section, we complete this picture by introducing $\sNCQSym$, the \emph{algebra of quasisymmetric functions in noncommuting variables in superspace}. Furthermore, we demonstrate that $\sNCQSym$ can also be realized as an algebra of invariants under a noncommuting analogue of the quasisymmetrizing action of the group $\sym_\infty$ in superspace. Subsequently, we endow $\sNCQSym$ with a Hopf algebra structure that simultaneously generalizes the Hopf structure of $\NCQSym$ given in \cite[Subsection 5.2]{BeZa09} and the Hopf structure of $\sQSym$ established in \cite[Subsection 5.2]{FiLaPi19}. Finally, we show that $\sNCSym$ naturally inherits a Hopf superalgebra structure by virtue of being a sub-Hopf superalgebra of $\sNCQSym$.

\subsection{The algebra of quasisymmetric functions in noncommuting variables in superspace}\label{036}

Let $\Q^\theta\dangle{x}$ be the algebra of formal power series of bounded degree in variables $x=(x_1,x_2,\ldots)$ and $\theta=(\theta_1,\theta_2,\ldots)$, subject to the relations $x_i\theta_j=\theta_jx_i$ and $\theta_i\theta_j=-\theta_j\theta_i$. Note that every monomial $u\in\Q^\theta\dangle{x}$ can be uniquely expressed in the normal form $u=q\theta_{i_1}\cdots\theta_{i_m}x_{j_1}\cdots x_{j_n}$, where $q\in\Q$ and $i_1<\cdots<i_m$. In what follows, we shall always assume that monomials are written in this form.

Analogously to Subsection~\ref{003}, for a monomial $u=\theta_{i_1}\cdots\theta_{i_m}x_{j_1}\cdots x_{j_n}$, the \emph{set of indices} of $u$ is defined as $\ind(u)=\{i_1,\ldots,i_m,j_1,\ldots,j_n\}$. Let $a\mapsto a^*$ be the unique order-preserving map from $\ind(u)$ to $[k]$, where $k=|\ind(u)|$. The \emph{standardization} of $u$ is the monomial defined by $\std(u)=\theta_{i_1^*}\cdots\theta_{i_m^*}x_{j_1^*}\cdots x_{j_n^*}$. We say that $u$ is \emph{standard} if $\std(u)=u$.

A formal power series $f\in\Q^\theta\dangle{x}$ is called a \emph{quasisymmetric function in noncommuting variables in superspace} if all the monomials with the same standardization have the same coefficient; that is, if $pu$ and $qv$ are two monomials occurring in $f$ such that $\std(u)=\std(v)$, then $p=q$. By definition, the set of all such quasisymmetric functions forms a subspace of $\Q^\theta\dangle{x}$ denoted by $\sNCQSym$.

The natural basis for $\sNCQSym$ is indexed by set supercompositions. A \emph{set supercomposition} $I$ is a finite sequence $(I_1,\ldots,I_k)$ of nonempty subsets of $[n]_0$, called \emph{blocks}, such that their union, after removing $0$ from every block, is $[n]$, and $I_i\cap I_j\subseteq\{0\}$ for all $i\neq j$. Blocks containing $0$ are called \emph{fermionic blocks}, and we say that $I$ is of \emph{bidegree} $(n,m)$ if it contains exactly $m$ such fermionic blocks. The set of all set supercompositions of bidegree $(n,m)$ is denoted by $\sComp_{n,m}$. For instance, below are two set supercompositions $I$ and $J$ of bidegree $(4,3)$:\[I=(\{0,1,3\},\{4\},\{0\},\{0,2\})\quad\text{and}\quad J=(\{4\},\{0,2\},\{0,1,3\},\{0\}).\]

For a monomial $u$ with $\std(u)=\theta_{i_1}\cdots\theta_{i_m}x_{j_1}\cdots x_{j_n}$, we denote by $I(u)$ the unique set supercomposition $(I_1,\ldots,I_k)$, where $k=|\ind(u)|$ and the $r$th block is given by $I_r=\{t\in[n]\mid j_t=r\}\cup\{0\mid r\in\{i_1,\ldots,i_m\}\}$. In other words, each block $I_r$ consists of the positions where the variable $x_r$ occurs in $\std(u)$, and the element $0$ is added whenever $\theta_r$ occurs in $\std(u)$. Note that, by definition, $I(u)=I(v)$ if and only if $\std(u)=\std(v)$. For instance, for $u=\theta_2\theta_8\,x_7x_2x_7x_5x_9x_2x_5x_7$, we have\[\std(u)=\theta_1\theta_4\,x_3x_1x_3x_2x_5x_1x_2x_3\quad\text{and}\quad I(u)=(\{0,2,6\},\{4,7\},\{1,3,8\},\{0\},\{5\}).\]The \emph{monomial quasisymmetric function in noncommuting variables in superspace} indexed by $I$ is defined as the formal power series:\[M_I=\sum_{I(u)=I}u.\]For instance, if $I=(\{2,4\},\{0,1,5\},\{0,3\})$, we have\[M_I=\theta_2\theta_3x_2x_1x_3x_1x_2+\theta_2\theta_4x_2x_1x_4x_1x_2+\theta_3\theta_4 x_3x_1x_4x_1x_3+\theta_3\theta_4x_3x_2x_4x_2x_3+\cdots\]

The subspace of $\sNCQSym$ spanned by the set $\{M_I\mid I\text{ is a set supercomposition of bidegree }(n,m)\}$ is denoted by $\sNCQSym_{n,m}$. This induces a $\Z_2$-grading that endows $\sNCQSym$ with the structure of a superalgebra:\begin{equation}\label{030}\sNCQSym=\sNCQSym_0\oplus\sNCQSym_1,\end{equation}where\[\sNCQSym_0=\bigoplus_{n,k\geq0}\sNCQSym_{n,2k}\quad\text{and}\quad\sNCQSym_1=\bigoplus_{n,k\geq0}\sNCQSym_{n,2k+1}.\]

The product of the basis functions $M_I$ can be described by an adaptation of the classical quasi-shuffle construction \cite{MHof00}. To this end, we first introduce a shuffle operation on set supercompositions.

For a set supercomposition $J=(J_1,\ldots,J_k)$ and an integer $n$, the \emph{$n$-shift} of $J$ is the sequence $J[n]=(J_1[n],\ldots,J_k[n])$, where each $J_i[n]$ is obtained from $J_i$ by adding $n$ to every positive element. Now, let $I=(I_1,\ldots,I_h)$ be a set supercomposition of bidegree $(n,m)$, and let $J=(J_1,\ldots,J_k)$ be another set supercomposition. An \emph{$(I,J)$-shuffle} is a set supercomposition $(K_1,\ldots,K_{h+k})$ such that the sequence of its blocks is a permutation of the blocks of $I$ and $J[n]$ that preserves the relative order of the blocks from each original sequence. More precisely, if $K_a=I_i$ and $K_b=I_j$ with $i<j$, then $a<b$; similarly, if $K_a=J_i[n]$ and $K_b=J_j[n]$ with $i<j$, then $a<b$. We denote by $\varepsilon(K)$ the number of inversions between the fermionic blocks of $I$ and $J$ in $K$, that is, $\varepsilon(K)=|\{(i,j)\mid I_i\text{ and }J_j[n]\text{ are fermionic, and }J_j[n]\text{ appears before }I_i\text{ in }K\}|$. For instance, if $I=(\{{\red0}\},\{{\red0},{\red1}\})$ and $J=(\{0,1,3\},\{2\})$, the $(I,J)$-shuffles are the following:\[\begin{array}{lll}(\{{\red0}\},\{{\red0},{\red1}\},\{0,2,4\},\{3\}),\quad&(\{{\red0}\},\{0,2,4\},\{{\red0},{\red1}\},\{3\}),\quad&(\{0,2,4\},\{{\red0}\},\{{\red0},{\red1}\},\{3\}),\\(\{{\red0}\},\{0,2,4\},\{3\},\{{\red0},{\red1}\}),\quad&(\{0,2,4\},\{{\red0}\},\{3\},\{{\red0},{\red1}\}),\quad&(\{0,2,4\},\{3\},\{{\red0}\},\{{\red0},{\red1}\}).\end{array}\]A set supercomposition $K$ is called an \emph{$(I,J)$-quasi-shuffle} if it can be obtained from an $(I,J)$-shuffle $L=(L_1,\dots,L_{h+k})$ by a possibly empty sequence of mergers of consecutive blocks $L_i$ and $L_{i+1}$ into their union $L_i\cup L_{i+1}$, with the restriction that $L_i\in I$, $L_{i+1}\in J$, and they cannot both be fermionic. The set of all $(I,J)$-quasi-shuffles is denoted by $\QSh(I,J)$. The \emph{sign} of an $(I,J)$-quasi-shuffle $K$ is defined as $\sgn(K)=(-1)^{\varepsilon(L)}$, where $L$ is the unique $(I,J)$-shuffle from which $K$ is obtained.

\begin{proposition}\label{001}
Let $I$ and $J$ be two set supercompositions. Then\[M_I\,M_J=\sum_{K\in\QSh(I,J)}\sgn(K)M_K.\]Consequently, $\sNCQSym$ is an algebra.
\end{proposition}
\begin{proof}
Let $u=\theta_{i_1}\cdots\theta_{i_m}x_{j_1}\cdots x_{j_n}$ and $v=\theta_{p_1}\cdots\theta_{p_r}x_{q_1}\cdots x_{q_s}$ be monomials occurring in $M_I$ and $M_J$, respectively. If $\{i_1,\ldots,i_m\}\cap\{p_1,\ldots,p_r\}\neq\emptyset$, then $uv=0$, since $\theta_i^2=0$. Thus, we may assume they are disjoint. We write $uv=\theta_{i_1}\cdots\theta_{i_m}\theta_{p_1}\cdots\theta_{p_r}
\,x_{j_1}\cdots x_{j_n}x_{q_1}\cdots x_{q_s}=\sgn(uv)w$, where $\sgn(uv)$ is a power of $-1$ obtained by reordering the fermionic variables into increasing order, and $w$ is a monic monomial.

Let $K=(K_1,\ldots,K_k)$ be the set supercomposition associated with $w$. Recall that each block $K_t$ is determined by an index $a_t\in\ind(w)$, and consists of the positions where the variable $x_{a_t}$ appears, together with the element $0$ whenever $\theta_{a_t}$ appears in $w$. Let $\ind(w)=\{a_1<\cdots<a_k\}$. We construct the sequence of blocks $K_t$ by considering each index $a_t$ in strictly increasing order: (1) If $a_t\in\ind(u)\setminus\ind(v)$, then all occurrences of $x_{a_t}$ come from $u$, and therefore $K_t=I_i$ for some $i$. (2) If $a_t\in\ind(v)\setminus\ind(u)$, then $K_t=J_j[n]$ for some $j$. (3) If $a_t\in\ind(u)\cap\ind(v)$, then the occurrences of $x_{a_t}$ come from both $u$ and $v$, and we obtain $K_t=I_i\cup J_j[n]$ for some $i$ and $j$. Note that this construction preserves the relative order of the blocks coming from $I$ and from $J[n]$, since the indices in $\ind(u)$ and $\ind(v)$ are considered in increasing order. Therefore, $K$ is obtained from an $(I,J)$-shuffle by possibly joining consecutive blocks. On the other hand, the sign obtained by reordering the fermionic variables coincides with $\sgn(K)$, because each transposition corresponds to an inversion between fermionic blocks coming from $I$ and $J[n]$ in $K$.

Conversely, given $K\in\QSh(I,J)$, one can reverse this construction to obtain monomials $u$ and $v$ contributing to $M_I$ and $M_J$, respectively, such that their product produces a monomial in $M_K$ with coefficient $\sgn(K)$.
\end{proof}
For instance, if $I=(\{{\red1},{\red2}\},\{{\red0}\})$ and $J=(\{0,2\},\{1,3\})$, we have
\[\begin{array}{rcl}
M_IM_J&\!=\!&M_{(\{{\red1},{\red2}\},\{{\red0}\},\{0,4\},\{3,5\})}-M_{(\{{\red1},{\red2}\},\{0,4\},\{{\red0}\},\{3,5\})}-M_{(\{0,{\red1},{\red2},4\},\{{\red0}\},\{3,5\})}-M_{(\{{\red1},{\red2}\},\{0,4\},\{{\red0},3,5\})}\\[0.1cm]&&\quad-M_{(\{0,{\red1},{\red2},4\},\{{\red0},3,5\})}-M_{(\{0,4\},\{{\red1},{\red2}\},\{{\red0}\},\{3,5\})}-M_{(\{0,4\},\{{\red1},{\red2}\},\{{\red0},3,5\})}-M_{(\{{\red1},{\red2}\},\{0,4\},\{3,5\},\{{\red0}\})}\\[0.1cm]&&\quad-M_{(\{0,{\red1},{\red2},4\},\{3,5\},\{{\red0}\})}-M_{(\{0,4\},\{{\red1},{\red2}\},\{3,5\},\{{\red0}\})}-M_{(\{0,4\},\{{\red1},{\red2},3,5\},\{{\red0}\})}-M_{(\{0,4\},\{3,5\},\{{\red1},{\red2}\},\{{\red0}\})}.
\end{array}\]

\subsection{Quasisymmetrizing action in noncommuting variables in superspace}\label{037}

As mentioned in \cite[Subsection 2.2]{NoTh10}, the algebra $\NCQSym$ can be obtained via a noncommutative analogue of the quasisymmetrizing action. Following this approach, we extend this action to superspace and show that the algebra $\sNCQSym$ is an algebra of invariants under this extended action of the finitary symmetric group $\sym_\infty$. As with $\sQSym$, this result suggests that the algebra $\sNCQSym$ is one of the invariant algebras belonging to the framework established in \cite{PreArGoMa26B}.

A \emph{pseudo set supercomposition} is a finite sequence $I=[I_1,\ldots,I_k]$ of subsets of $[n]_0$, called \emph{blocks}, such that the subsequence $(I_{i_1},\ldots,I_{i_h})$ formed by its nonempty blocks defines a set supercomposition. As with $\Q^\theta\dbrack{x}$, we associate every monic monomial in $\Q^\theta\dangle{x}$ with a pseudo set supercomposition. For a monomial $u=\theta_{i_1}\cdots\theta_{i_m}x_{j_1}\cdots x_{j_n}$ with $i_1<\cdots<i_m$ and $k=\max(\ind(u))$, its associated pseudo set supercomposition is $\tilde{I}(u)=[I_1,\ldots,I_k]$, where the $r$th block is $I_r=\{t\in[n]\mid j_t=r\}\cup\{0\mid r\in\{i_1,\ldots,i_m\}\}$ for all $r\in[k]$. Note that the subsequence formed by the nonempty blocks of $\tilde{I}(u)$ is precisely $I(u)$, the associated set supercomposition of $u$. For instance, if $u=\theta_2\theta_8x_7x_2x_7x_5x_9x_2x_5x_7$, we have
\[\tilde{I}(u)=[\emptyset,\{0,2,6\},\emptyset,\emptyset,\{4,7\},\emptyset,\{1,3,8\},\{0\},\{5\}].\]

These assignments establish a bijection between the set of monic monomials and the set of all pseudo set supercompositions. Using this bijection, we can define an action of the finitary symmetric group $\sym_\infty$ on the monic monomials, which then extends naturally to the entire algebra $\Q^\theta\dangle{x}$. Recall that $s_i$ denotes the simple transposition exchanging $i$ with $i+1$. We define the \emph{quasisymmetrizing action in noncommuting variables} of $s_i$ on a pseudo set supercomposition $I=[I_1,\ldots,I_k]$ by\[s_iI=\begin{cases}
\,[I_1,\ldots,I_{i-1},I_{i+1},I_i,I_{i+2},\ldots,I_k]&\text{if }i<k\text{ and }\emptyset\in\{I_i,I_{i+1}\}\\
\,[I_1,\ldots,I_{k-1},0,I_k]&i=k\\
\,[I_1,\ldots,I_k]&\text{otherwise}.
\end{cases}\]
It is straightforward to verify that the operators $s_i$ satisfy the Coxeter relations, ensuring that this indeed defines an action of $\sym_\infty$ on $\Q^\theta\dangle{x}$. For instance, if $u=\theta_2\theta_4x_3^2x_2x_6x_3x_2x_6$ and $\sigma=s_4s_3s_6s_1$, we have\[\sigma u=\sigma\,\theta_{\blue2}\theta_{\red4}x^2_{\red3}x_{\blue2}x_{\blue6}x_{\red3}x_{\blue2}x_{\blue6}=s_4\theta_1\theta_{\blue4}x_3^2x_1x_7x_3x_1x_7=\theta_1\theta_5x_3^2x_1x_7x_3x_1x_7.\]

Given a set supercomposition $I=(I_1,\ldots,I_k)$ and a finite subset $A\subset\N$ with $|A|=|(I_1\cup\cdots\cup I_k)\setminus\{0\}|$, we denote by $A^I$ the unique monic monomial whose set of indices is $A$ and whose associated set supercomposition is $I$. For instance, the monomial $u=\theta_2\theta_4x_3^2x_2x_6x_3x_2x_6$ is defined by the pseudo set supercomposition $[{\red\emptyset},\{0,3,6\},\{1,2,5\},\{0\},{\red\emptyset},\{4,7\}]$, and so $u=\{2,3,4,6\}^{(\{0,3,6\},\{1,2,5\},\{0\},\{4,7\})}$.

The following description is analogous to Proposition~\ref{016}.
\begin{proposition}\label{026}
Let $I$ be a set supercomposition, $A\subset\N$ be a finite subset with $|A|=|(I_1\cup\cdots\cup I_k)\setminus\{0\}|$, and $\sigma\in\sym_\infty$. Then the quasi-symmetrizing action in noncommuting variables on the monomial $A^I$ is given by $\sigma\cdot A^I=\sigma(A)^I$, where $\sigma(A)=\{\sigma(a)\mid a\in A\}$.
\end{proposition}
\begin{proof}
It suffices to verify the statement for a simple transposition $\sigma=s_i$. Let $u=A^I$ and let $\tilde{I}(u)=[I_1,\ldots,I_k]$. Note that $i\in A$ if and only if $I_i\neq\emptyset$. By definition, the quasisymmetrizing action of $\sigma$ depends only on the blocks $I_i$ and $I_{i+1}$, where we conventionally set $I_{i+1}=\emptyset$ if $i\geq k$. If exactly one of $I_i$ or $I_{i+1}$ is empty, then $\sigma$ interchanges these two blocks, which corresponds to replacing $i$ by $i+1$ or vice versa in the support set $A$. If both are empty or both are nonempty, then the action is trivial and $A$ is unchanged. In all cases, we have $\sigma u=\sigma(A)^I$. Since the simple transpositions generate $\sym_\infty$, the result follows.
\end{proof}
For instance, if $u=\theta_2\theta_4x_3^2x_2x_6x_3x_2x_6$ and $\sigma=s_4s_3s_6s_1$, we have\[\sigma u=\sigma(\{2,3,4,6\})^{(\{0,3,6\},\{1,2,5\},\{0\},\{4,7\})}=\{1,3,5,7\}^{(\{0,3,6\},\{1,2,5\},\{0\},\{4,7\})}=\theta_1\theta_5x_3^2x_1x_7x_3x_1x_7.\]

Finally, we show that the quasisymmetrizing action in noncommuting variables provides a complete characterization of the algebra of quasisymmetric functions in noncommuting variables in superspace.
\begin{proposition}\label{027}
A formal power series $f$ of bounded degree in $\Q^\theta\dangle{x}$ is a quasisymmetric function in noncommuting variables in superspace if and only if $\sigma \cdot f=f$ for all permutations $\sigma\in\sym_\infty$.
\end{proposition}
\begin{proof}
Given $f\in \Q^\theta\dangle{x}$, we can write\[f=\sum_I\sum_AC_{A,I}\,A^I,\]where $I$ runs over set supercompositions and $A$ over finite subsets of $\N$ with $|A|=|(I_1\cup\cdots\cup I_k)\setminus\{0\}|$ if $I=(I_1,\ldots,I_k)$. By Proposition~\ref{026}, we have $\sigma\cdot A^I=\sigma(A)^I$, so the action of $\sym_\infty$ preserves $I$ and acts only on $A$. If $f$ belongs to $\sNCQSym$, then the coefficients $C_{A,I}$ depend only on $I$. Since $\sigma$ acts as a bijection on the finite subsets of $\N$ satisfying $|A|=|(I_1\cup\cdots\cup I_k)\setminus\{0\}|$, it simply permutes the terms in the inner sum, leaving the function invariant. Hence, $\sigma f=f$. Conversely, if $\sigma f=f$ for all $\sigma\in\sym_\infty$, then for any fixed $I$ and any $A,B$ with $|A|=|B|=|(I_1\cup\cdots\cup I_k)\setminus\{0\}|$, there exists a permutation $\sigma\in\sym_\infty$ such that $\sigma(A)=B$, and thus $C_{A,I}=C_{B,I}$. Therefore, $f$ belongs to $\sNCQSym$.
\end{proof}

\subsection{Hopf superalgebra of quasisymmetric functions in noncommuting variables in superspace}\label{038}

We now endow $\sNCQSym$ with a Hopf superalgebra structure. As is customary in the theory of combinatorial Hopf algebras, we define the coproduct by extending functions to duplicated alphabets. 

Let $\Q^{\theta,\vartheta}\dangle{x,y}$ be the algebra of formal power series of bounded degree in two families of noncommuting variables $x=(x_1,x_2,\ldots)$ and $y=(y_1,y_2,\ldots)$ together with two families of anticommuting variables $\theta=(\theta_1,\theta_2,\ldots)$ and $\vartheta=(\vartheta_1,\vartheta_2,\ldots)$, subject to the following relations:
\[\theta_i\vartheta_j=-\vartheta_j\theta_i,\qquad\theta_ix_j=x_j\theta_i,\qquad\vartheta_ix_j=x_j\vartheta_i,\qquad\theta_iy_j=y_j\theta_i,\qquad \vartheta_iy_j=y_j\vartheta_i.\]
These defining relations, together with the fact that the countably infinite sets of variables $x$ and $\theta$ have the same cardinalities as the disjoint unions $x\cup y$ and $\theta\cup\vartheta$ respectively, imply that the algebras of formal power series $\Q^\theta\dangle{x}$ and $\Q^{\theta,\vartheta}\dangle{x,y}$ are isomorphic.

Since quasisymmetric functions in $\Q^\theta\dangle{x}$ are defined, whether via standardization or the quasisymmetrizing action, with respect to the total order of the variable families $x$ and $\theta$, we must first totally order the duplicated alphabets. Specifically, we endow the sets $x\cup y$ and $\theta\cup\vartheta$ with the unique total order that extends the natural index orders of the individual alphabets and satisfies $x_i<y_j$ and $\theta_i<\vartheta_j$ for all $i,j$. We denote by $\sNCQSym(x,y;\theta,\vartheta)$ the subalgebra of $\Q^{\theta,\vartheta}\dangle{x,y}$ consisting of quasisymmetric functions in noncommuting variables in superspace. By definition, there is a natural isomorphism
\[\sNCQSym\to\sNCQSym(x,y;\theta,\vartheta),\quad M_I\mapsto M_I(x,y;\theta,\vartheta),\]
where $M_I(x,y;\theta,\vartheta)$ denotes the monomial quasisymmetric function in noncommuting variables in superspace evaluated over the totally ordered duplicated alphabet.

Let $\Q^{\theta,\vartheta}\dangle{x\times y}$ be the quotient of $\Q^{\theta,\vartheta}\dangle{x,y}$ by the relations $x_iy_j=y_jx_i$ for all $i,j$. This quotient is naturally isomorphic to the tensor product superalgebra $\Q^\theta\dangle{x}\otimes\Q^\vartheta\dangle{y}$ equipped with the super tensor product multiplication:\[(f_1\otimes g_1)(f_2\otimes g_2)=(-1)^{|f_2||g_1|}f_1f_2\otimes g_1g_2,\]where $|f|$ denotes the parity of a homogeneous element $f$ with respect to the $\Z_2$-grading defined in~\eqref{030}.

Now, let $\sNCQSym^\otimes$ be the image of $\sNCQSym(x,y;\theta,\vartheta)$ under the canonical projection $\Q^{\theta,\vartheta}\dangle{x,y}\twoheadrightarrow\Q^{\theta,\vartheta}\dangle{x\times y}$. By definition, $\sNCQSym^\otimes$ can be regarded as a subsuperalgebra of $\Q^\theta\dangle{x}\otimes\Q^\vartheta\dangle{y}$. We define the coproduct $\Delta$ of $\sNCQSym$ as the composition
\[\Delta:\sNCQSym\to\sNCQSym(x,y;\theta,\vartheta)\twoheadrightarrow\sNCQSym^\otimes\simeq\sNCQSym\otimes\sNCQSym.\]

Before formally establishing the Hopf superalgebra structure of $\sNCQSym$, we provide an explicit formula for its coproduct evaluated on the monomial basis.
\begin{proposition}\label{028}
Let $I=(I_1,\ldots,I_k)$ be a set supercomposition. Then\[\Delta(M_I)=\sum_{i=0}^kM_{\std(I_1,\ldots,I_i)}\otimes M_{\std(I_{i+1},\ldots,I_k)}.\]
\end{proposition}
\begin{proof}
Let $u\in\Q^{\theta,\vartheta}\dangle{x,y}$ be a monomial occurring in $M_I(x,y;\theta,\vartheta)$. Since the sets of variables are totally ordered with $x_i<y_j$ and $\theta_i<\vartheta_j$ for all $i,j$, there exists a unique index $i\in[k]_0$ such that the blocks $I_1,\ldots,I_i$ correspond to variables in $x\cup\theta$, and the remaining blocks $I_{i+1},\ldots,I_k$ correspond to variables in $y\cup\vartheta$. Under the canonical projection onto $\Q^{\theta,\vartheta}\dangle{x\times y}$, any such monomial $u$ can be uniquely written as a concatenation $u=u_1u_2$, where $u_1\in\Q^\theta\dangle{x}$ is a monomial occurring in $M_{\std(I_1,\ldots,I_i)}$, and $u_2\in\Q^\vartheta\dangle{y}$ is a monomial occurring in $M_{\std(I_{i+1},\ldots,I_k)}$. Thus, $u$ is mapped to the tensor $u_1\otimes u_2$. Conversely, any tensor $u_1\otimes u_2$, where $u_1$ and $u_2$ are as above, corresponds to a unique monomial $u=u_1u_2$ contributing to $M_I(x,y;\theta,\vartheta)$. Summing over all possible splitting indices $i$, we obtain the result.
\end{proof}

From Proposition~\ref{028}, it follows immediately that $\Delta$ is coassociative. We thus obtain the following structural result.
\begin{theorem}\label{032}
The algebra $\sNCQSym$, equipped with the usual product and the coproduct $\Delta$, is a graded Hopf superalgebra, with the usual counit and antipode.
\end{theorem}

\subsection{The sub-Hopf superalgebra $\sNCSym$}\label{039}

Similarly to a symmetric function in superspace, a formal power series $f=f(x,\theta)\in\Q^\theta\dangle{x}$ is called a \emph{symmetric function in noncommuting variables in superspace} if it is invariant under any simultaneous permutation of the indices of both $x$ and $\theta$; that is, $f(x,\theta)=f(\sigma x,\sigma\theta)$ for all $\sigma\in\sym_n$ and all $n\in\N$ \cite[Subsection 3.1]{ArGoMa25}. The set of all such symmetric functions forms a subalgebra of $\Q^\theta\dangle{x}$ called the \emph{algebra of symmetric functions in noncommuting variables in superspace}, denoted by $\sNCSym$ \cite[Subsection 3.2]{ArGoMa25}. This algebra naturally contains the algebra of symmetric functions in noncommuting variables $\NCSym$ \cite{DeLaMa03,DeLaMa04} as the subalgebra of elements that are independent of the variables $\theta$.

A set supercomposition $I=(I_1,\ldots,I_k)$ is called a \emph{set superpartition} if $I_i\neq I_j$ and $\min(I_i\setminus I_j) <\min(I_j\setminus I_i)$ for all $i,j\in [k]$ with $i<j$, under the convention that $\min(\emptyset)=0$. Note that the distinctness condition $I_i\neq I_j$ ensures there can be at most one block equal to $\{0\}$. Moreover, since $0$ belongs exclusively to the fermionic blocks, this inequality forces all fermionic blocks to appear strictly before the non-fermionic ones.

The \emph{monomial symmetric function in noncommuting variables in superspace} indexed by a set superpartition $I=(I_1,\ldots,I_k)$ is defined as
\[m_I=\sum_{\sigma\in\sym_k}(-1)^{\inv_I(\sigma)}M_{I^\sigma},\]where $I^\sigma=(I_{\sigma(1)},\ldots,I_{\sigma(k)})$, and $\inv_I(\sigma)$ is the number of pairs $(p,q)$ such that $p<q$, both $I_p$ and $I_q$ are fermionic blocks, and $\sigma^{-1}(p)>\sigma^{-1}(q)$. For instance, if $I=(\{0,2,4\},\{0,3\},\{1\})$, we have\[\begin{array}{rcl}m_I&=&M_{(\{0,2,4\},\{0,3\},\{1\})}+M_{(\{0,2,4\},\{1\},\{0,3\})}-M_{(\{0,3\},\{0,2,4\},\{1\})}\\[0.1cm]&&\,-\,M_{(\{0,3\},\{1\},\{0,2,4\})}+M_{(\{1\},\{0,2,4\},\{0,3\})}-M_{(\{1\},\{0,3\},\{0,2,4\})}.\end{array}\]Note that $m_I$ is a symmetric function in noncommuting variables in superspace. Indeed, the set $\{m_I\mid I\text{ is a set superpartition}\}$ forms the monomial basis of $\sNCSym$ \cite[Proposition~3.4]{ArGoMa25}. Moreover, following \cite[Definition~3.2]{ArGoMa25}, if $I$ is of bidegree $(n,m)$, then $m_I$ can also be described explicitly as
\[m_I=\sum_{(i_1,\dots,i_m,j_1,\ldots,j_n)} \theta_{i_1}\cdots \theta_{i_m}x_{j_1}\cdots x_{j_n},\]
where the sum runs over all indices satisfying:
\begin{enumerate}
\item For each $s\in [m]$, $i_s=j_t$ if and only if $t\in I_{s}$.
\item For each $s,t\in [n]$, $j_s=j_t$ if and only if $s,t$ belong to the same block of $I$.
\end{enumerate}

The product of the monomial symmetric functions in noncommuting variables in superspace was originally described in \cite[Subsection~3.3]{ArGoMa25}. We restate this rule in terms of set superpartitions as follows. Let $I$ be a set superpartition of bidegree $(n,m)$ and let $J$ be a set superpartition. The product $m_I\,m_J$ is obtained by summing over all set superpartitions arising from admissible fusions between the blocks of $I$ and those of the shifted superpartition $J[n]$, with the restriction that no two fermionic blocks may be merged. Each resulting term is subsequently reordered to satisfy the standard definition of set superpartitions. The sign of each term is determined by the number of inversions of the fermionic blocks generated during this reordering process. For instance, if $I = (\{0\},\{0,3\},\{1,2\})$ and $J=(\{{\blue0},{\blue2}\},\{{\blue1}\})$, we have
\[\begin{array}{rcl}
m_I\,m_J&=&m_{(\{0\},\{0,3\},\{{\blue0},{\blue5}\},\{1,2\},\{{\blue4}\})}-m_{(\{0,3\},\{0,{\blue4}\},\{{\blue0},{\blue5}\},\{1,2\})}+m_{(\{0\},\{0,3,{\blue4}\},\{{\blue0},{\blue5}\},\{1,2\})}\\[0.1cm]
&&\,+\,\,m_{(\{0\},\{0,3\},\{{\blue0},{\blue5}\},\{1,2,{\blue4}\})}+m_{(\{0\},\{0,3\},\{{\blue0},1,2,{\blue5}\},\{{\blue4}\})}-m_{(\{0,3\},\{0,{\blue4}\},\{{\blue0},1,2,{\blue5}\})}\\[0.1cm]
&&\,+\,\,m_{(\{0\},\{0,3,{\blue4}\},\{{\blue0},1,2,{\blue5}\})}-m_{(\{0\},\{{\blue0},{\blue5}\},\{0,3,{\blue4}\},\{1,2\})}-m_{(\{0\},\{{\blue0},{\blue5}\},\{0,1,2,3,{\blue4}\})}.
\end{array}\]

We now describe the coproduct of $\sNCSym$ in this monomial basis. For a set superpartition $I=(I_1,\ldots,I_k)$, a subset $A=\{i_1<\cdots<i_s\}\subseteq[k]$, and a permutation $\sigma$ of $A$, we set $I_A=(I_{i_1},\ldots,I_{i_s})$ and $I_A^\sigma=(I_{\sigma(i_1)},\ldots,I_{\sigma(i_s)})$.
\begin{proposition}\label{029}
Let $I=(I_1,\ldots,I_k)$ be a set superpartition. Then
\[\Delta(m_I)=\sum_{A\subseteq [k]}(-1)^{\inv_I(A)}\,m_{\std(I_A)}\otimes m_{\std(I_{A^c})},\]
where $\inv_I(A)$ is the number of pairs $(p,q)$ such that $p\in A^c$, $q\in A$, $p<q$, and both $I_p$ and $I_q$ are fermionic.
\end{proposition}
\begin{proof}
Applying the coproduct formula in Proposition~\ref{028} to the monomial basis, we obtain
\[\Delta(m_I)=\sum_{\sigma\in\sym_k}(-1)^{\inv_I(\sigma)}\Delta(M_{I^\sigma})=\sum_{\sigma\in\sym_k}(-1)^{\inv_I(\sigma)}\sum_{s=0}^kM_{\std(I_{\sigma(1)},\ldots,I_{\sigma(s)})}\otimes M_{\std(I_{\sigma(s+1)},\ldots,I_{\sigma(k)})}.\]
Grouping the permutations according to the subset of original indices that appear in the first tensor factor, each pair $(\sigma,s)$ uniquely determines a subset $A\subseteq[k]$ of size $s$, together with permutations $\sigma_A$ of $A$ and $\sigma_{A^c}$ of $A^c$, such that
\[\begin{array}{rcl}
\Delta(m_I)&=&{\displaystyle\sum_{A\subseteq[k]}(-1)^{\inv_I(A)}\sum_{(\sigma_A,\sigma_{A^c})}(-1)^{\inv_{I_{\!A}}(\sigma_A)}M_{\std(I_A^{\sigma_A})}\otimes(-1)^{\inv_{I_{\!A^c}}(\sigma_{A^c})}M_{\std(I_{A^c}^{\sigma_{A^c}})}}\\[0.5cm]
&=&{\displaystyle\sum_{A\subseteq[k]}(-1)^{\inv_I(A)}\underbrace{\Bigl(\,\sum_{\sigma_A}(-1)^{\inv_{I_{\!A}}(\sigma_A)}M_{\std(I_A^{\sigma_A})}\Bigr)}_{m_{\std(I_A)}}\otimes\underbrace{\Bigl(\,\sum_{\sigma_{A^c}}(-1)^{\inv_{I_{\!A^c}}(\sigma_{A^c})}M_{\std(I_{A^c}^{\sigma_{A^c}})}\Bigr)}_{m_{\std(I_{A^c})}}},
\end{array}\]
where the second sum is over all pairs $(\sigma_A,\sigma_{A^c})$ such that $\sigma_A$ is a permutation of $A$ and $\sigma_{A^c}$ is a permutation of $A^c$. This proves the claim.
\end{proof}
For instance, if $I=(\{0,2,4\},\{0,3\},\{1\})$, we have\[\begin{array}{rcl}\Delta(m_I)
&=&1\otimes m_{(\{0,2,4\},\{0,3\},\{1\})}+m_{(\{0,1,2\})}\otimes m_{(\{0,2\},\{1\})}-m_{(\{0,1\})}\otimes m_{(\{0,2,3\},\{1\})}\\[0.1cm]
&&\,+\,m_{(\{1\})}\otimes m_{(\{0,1,3\},\{0,2\})}+m_{(\{0,1,3\},\{0,2\})}\otimes m_{(\{1\})}+m_{(\{0,2,3\},\{1\})}\otimes m_{(\{0,1\})}\\[0.1cm]
&&\,-\,m_{(\{0,2\},\{1\})}\otimes m_{(\{0,1,2\})}+m_{(\{0,2,4\},\{0,3\},\{1\})}\otimes 1.\end{array}\]

Finally, by Proposition~\ref{029}, we obtain the following result.
\begin{proposition}
The algebra of symmetric functions in noncommuting variables in superspace $\sNCSym$ is a sub-Hopf superalgebra of $\sNCQSym$.
\end{proposition}

\section{The $Q$-basis of $\sNCQSym$}\label{024}

In this section, we introduce a new basis for $\sNCQSym$, analogous to the $Q$-basis of $\NCQSym$ defined in \cite[Section 6]{BeZa09}. To this end, we first define a partial order on set supercompositions that generalizes the order established in \cite[Section 6]{BeZa09}. Subsequently, we study the product and coproduct rules for this $Q$-basis. Just as the classical Malvenuto--Reutenauer algebra of permutations \cite[Section 3]{MaRe95} can be obtained as the subalgebra generated by the minimal elements of $\NCQSym$ \cite{DuHiTh02} \cite[Section 6]{BeZa09} \cite[Subsection 2.4]{NoThWi10}, we introduce a Malvenuto--Reutenauer algebra in superspace, defined as the sub-Hopf superalgebra generated by the minimal elements of $\sNCQSym$. Finally, via the projection from $\sNCQSym$ onto $\sQSym$, we establish a formula to compute the product of fundamental quasisymmetric functions in superspace (Theorem~\ref{002}).

\subsection{A partial order on set supercompositions}\label{005}

Given two set supercompositions, we say that $J$ \emph{covers} $I=(I_1,\ldots,I_k)$ if there exists an index $i\in[k-1]$ such that neither $I_i$ nor $I_{i+1}$ is fermionic, $\max(I_i)<\min(I_{i+1})$, and $J=(I_1,\ldots,I_{i-1},I_i\cup I_{i+1},I_{i+2},\ldots,I_k)$. The set of all set supercompositions is partially ordered by the reflexive and transitive closure of this covering relation, denoted by $\preceq$. See Figure~\ref{004}. A sequence $(I_p,\ldots,I_q)$ of consecutive non-fermionic blocks of $I$ satisfying $\max(I_i)<\min(I_{i+1})$ for all $i\in[p,q-1]$ is called \emph{increasing}. By definition, if $J\succeq I$, then each non-fermionic block of $J$ is the union of the blocks within some increasing sequence of $I$, and the fermionic blocks of $J$ are exactly those of $I$. Note that, by restricting this relation to set supercompositions with no fermionic blocks, we naturally obtain the partial order on set compositions given in \cite[Section 6]{BeZa09}, which differs from the classical refinement order \cite[Subsections~1.4.3 and~6.2.3]{AguSwa06}.

\begin{figure}[H]\centering
\begin{tikzpicture}[node distance=2cm]
\node(a)at(+0,4.5){$_{(\{10\},\{{\red3},{\red4}\},\{0\},\{9\},\{0,1,5,7\},\{{\blue2},{\blue6},{\blue8}\})}$};
\node(b)at(-5.5,3.0){$_{(\{10\},\{{\red3},{\red4}\},\{0\},\{9\},\{0,1,5,7\},\{{\blue2},{\blue6}\},\{{\blue8}\})}$};
\node(c)at(+0,3.0){$_{(\{10\},\{{\red3},{\red4}\},\{0\},\{9\},\{0,1,5,7\},\{{\blue2}\},\{{\blue6},{\blue8}\})}$};
\node(d)at(+5.5,3.0){$_{(\{10\},\{{\red3}\},\{{\red4}\},\{0\},\{9\},\{0,1,5,7\},\{{\blue2},{\blue6},{\blue8}\})}$};
\node(e)at(-5.5,1.5){$_{(\{10\},\{{\red3},{\red4}\},\{0\},\{9\},\{0,1,5,7\},\{{\blue2}\},\{{\blue6}\},\{{\blue8}\})}$};
\node(f)at(+0,1.5){$_{(\{10\},\{{\red3}\},\{{\red4}\},\{0\},\{9\},\{0,1,5,7\},\{{\blue2},{\blue6}\},\{{\blue8}\})}$};
\node(g)at(+5.5,1.5){$_{(\{10\},\{{\red3}\},\{{\red4}\},\{0\},\{9\},\{0,1,5,7\},\{{\blue2}\},\{{\blue6},{\blue8}\})}$};
\node(h)at(+0,0){$_{(\{10\},\{{\red3}\},\{{\red4}\},\{0\},\{9\},\{0,1,5,7\},\{{\blue2}\},\{{\blue6}\},\{{\blue8}\})}$};
\draw(a)--(b);\draw(a)--(c);\draw(a)--(d);\draw(b)--(f);\draw(c)--(g);\draw(d)--(g);\draw(f)--(h);\draw(g)--(h);\draw(b)--(e);\draw(c)--(e);\draw(c)--(e);\draw(d)--(f);\draw(e)--(h);
\end{tikzpicture}
\caption{Elements greater than or equal to $(\{10\},\{3\},\{4\},\{0\},\{9\},\{0,1,5,7\},\{2\},\{6\},\{8\})$.}\label{004}
\end{figure}
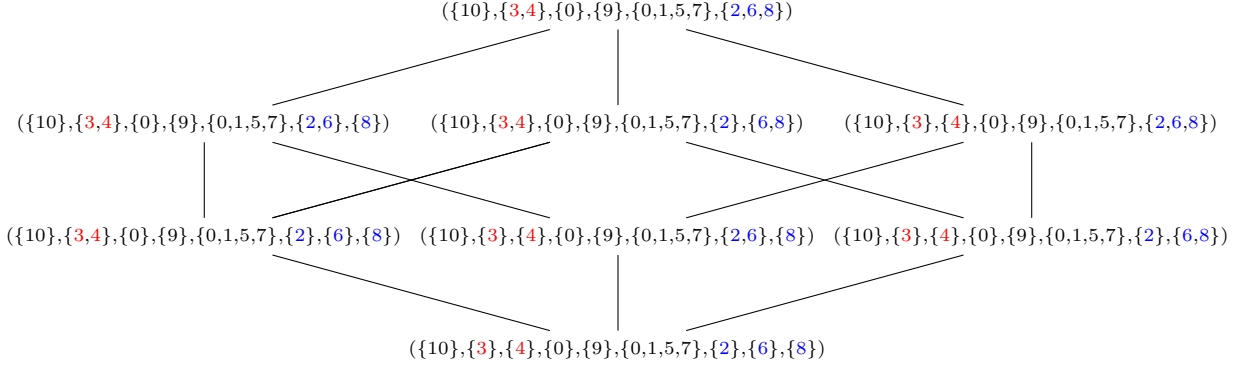
Observe that the minimal elements of this poset are precisely the set supercompositions in which every non-fermionic block is a singleton. We call these minimal elements \emph{superpermutations}. This terminology is naturally justified because if a superpermutation has no fermionic blocks, it can be canonically identified with an ordinary permutation simply by reading its singleton blocks from left to right. In this sense, superpermutations extend the classical notion of permutations to the superspace setting. For the sake of brevity, when a superpermutation has no fermionic blocks, we shall simply refer to it as a \emph{permutation}. The set of all superpermutations is denoted by $\sym$, and the subset of those having bidegree $(n,m)$ is denoted by $\sym_{n,m}$. In particular, $\sym_{n,0}$ coincides with the classical symmetric group $\sym_n$. For any superpermutation $I$, we set $I^{\,\uparrow}=\{J\mid I\preceq J\}$.

Given a superpermutation $I=(I_1,\ldots,I_k)$, a \emph{non-fermionic segment} is a maximal interval of indices $[p,q]\subseteq[k]$ such that $I_i$ is a non-fermionic block for all $i\in[p,q]$. For each non-fermionic segment $[p,q]$, we set $\sigma_{[p,q]}=(a_p,\ldots,a_q)$, where $I_i=\{a_i\}$ for all $i\in[p,q]$. Note that $\sigma_{[p,q]}$ can be naturally viewed as a permutation of its underlying set. We denote by $\alpha_{[p,q]}$ the classical composition determined by the descent set of this permutation \cite[Section~7.19]{St99}. Specifically, if $\Des(\sigma_{[p,q]})=\{i\mid a_{p+i-1}>a_{p+i}\}=\{d_1<\dots<d_s\}$ is nonempty, then $\alpha_{[p,q]}=(d_1,d_2-d_1,\ldots,d_s-d_{s-1},(q-p+1)-d_s)$, otherwise $\alpha_{[p,q]}=(q-p+1)$.

The \emph{dotted composition of a superpermutation} $I$, denoted by $\gamma(I)$, is formed by concatenating, in their natural order of appearance, the dotted parts $(\dot{a})$ with $a=\lvert I_i\rvert-1$ for each fermionic block $I_i$, and the classical compositions $\alpha_{[p,q]}$ for each non-fermionic segment $[p,q]$. Note that if $I$ has no fermionic blocks, then there is a unique non-fermionic segment $[k]$, and the construction of $\gamma(I)$ coincides with the classical association between a permutation and the composition determined by its descents \cite[Section~7.19]{St99}.

For instance, the superpermutation $I=(\{0,4,8\},\{{\red2}\},\{{\red5}\},\{{\red3}\},\{0\},\{{\blue6}\},\{{\blue4}\},\{{\blue7}\})$ has non-fermionic segments $[2,4]$ and $[6,8]$, with associated sequences $\sigma_{[2,4]}=(2,5,3)$ and $\sigma_{[6,8]}=(6,4,7)$. Since the descent sets of these sequences are $\Des(\sigma_{[2,4]})=\{2\}$ and $\Des(\sigma_{[6,8]})=\{1\}$, their corresponding compositions are $\alpha_{[2,4]}=(2,1)$ and $\alpha_{[6,8]}=(1,2)$, respectively. Thus,\[\gamma(I)=(\dot{2},{\red2},{\red1},\dot{0},{\blue1},{\blue2}).\]

We now introduce a map that assigns to each set supercomposition a dotted composition. The \emph{dotted composition associated with a set supercomposition} $I=(I_1,\ldots,I_k)$ is defined as $\alpha(I)=(\alpha_1,\ldots,\alpha_k)$, where, for each $i \in [k]$, we let $a_i=|I_i\setminus\{0\}|$ and set $\alpha_i=a_i$ if $0\notin I_i$, and $\alpha_i=\dot{a}_i$ if $0\in I_i$. Note that if $I$ is a classical set composition, then $\alpha(I)=(|I_1|,\ldots,|I_k|)$. Furthermore, if $J$ is another set composition such that $I\preceq J$, then $\alpha(I)\preceq\alpha(J)$. For instance, given $I=(\{0\},\{{\red3},{\red5}\},\{0,2,4\},\{0,1\})$, we have $\alpha(I)=(\dot{0},{\red2},\dot{2},\dot{1})$.

\begin{proposition}\label{007}
If $I$ is a permutation, then the map $J\mapsto\alpha(J)$ defines an isomorphism between the poset $I^{\,\uparrow}$ of set compositions above $I$ and the poset $\gamma(I)^{\,\downarrow}$ of integer compositions refining $\gamma(I)$. See Figure~\ref{008}.
\end{proposition}
\begin{proof}
Let $I=(\{a_1\},\ldots,\{a_n\})$ with $\Des(I)=\{i_1,\ldots,i_{k-1}\}$. Observe that the maximal element in $I^{\,\uparrow}$ is $K=(\{a_1,\ldots,a_{i_1}\},\ldots,\{a_{i_{k-1}+1},\ldots,a_n\})$. Since $I$ is a permutation, by definition $\alpha(I)=(1,\ldots,1)$, the minimal element in $\gamma(I)^{\,\downarrow}$, and $\alpha(K)=(b_1,\ldots,b_k)=\gamma(I)$, where $b_j=i_j-i_{j-1}$ with $i_0=0$ and $i_k=n$. As mentioned in \cite[Section~5]{BeZa09}, the poset $I^{\,\uparrow}$ is a Boolean lattice of rank $d:=|\Asc(\sigma_{[1,n]})|$, where $\Asc(\sigma_{[1,n]})=\{i\mid a_i<a_{i+1}\}$, the set of ascents of $\sigma_{[1,n]}$. Note that all ascents of $\sigma_{[1,n]}$ occur strictly within the blocks of $K$, and so $d=(b_1-1)+\cdots+(b_k-1)$. On the other hand, Proposition~\ref{006} implies that $\gamma(I)^{\,\downarrow}$ is a Boolean lattice of rank $(b_1-1)+\cdots+(b_k-1)$. Since $J\preceq J'$ implies $\alpha(J)\preceq\alpha(J')$, and both $I^{\,\uparrow}$ and $\gamma(I)^{\,\downarrow}$ are Boolean lattices of the same rank, the map $J\mapsto\alpha(J)$ defines an isomorphism.
\end{proof}
\begin{figure}[H]\centering
\begin{tikzpicture}[node distance=2cm]
\node(a)at(0,2){$_{(\{{\red1},{\red2},{\red4}\},\{3\})}$};
\node(b)at(-1.5,1){$_{(\{{\red1},{\red2}\},\{{\red4}\},\{3\})}$};
\node(c)at(+1.5,1){$_{(\{{\red1}\},\{{\red2},{\red4}\},\{3\})}$};
\node(d)at(0,0){$_{(\{{\red1}\},\{{\red2}\},\{{\red4}\},\{3\})}$};
\node(e)at(3.2,1){$\to$};
\node(f)at(6,2){$_{({\red3},1)}$};
\node(g)at(+4.5,1){$_{({\red2},{\red1},1)}$};
\node(h)at(+7.5,1){$_{({\red1},{\red2},1)}$};
\node(i)at(6,0){$_{({\red1},{\red1},{\red1},1)}$};
\draw(a)--(b);\draw(a)--(c);\draw(b)--(d);\draw(c)--(d);
\draw(f)--(g);\draw(f)--(h);\draw(g)--(i);\draw(h)--(i);
\end{tikzpicture}
\caption{Isomorphism between $I^{\,\uparrow}$ and $\gamma(I)^{\,\downarrow}$ for $I=(\{1\},\{2\},\{4\},\{3\})$.}\label{008}
\end{figure}
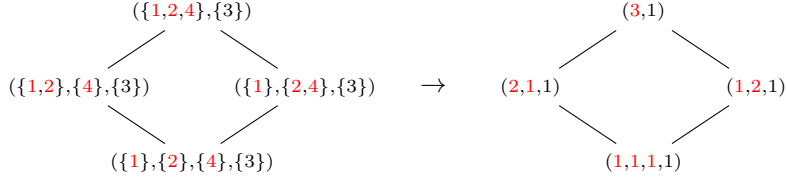

\begin{proposition}\label{009}
If $I$ is a superpermutation, then the map $J\mapsto\alpha(J)$ defines an isomorphism between the poset $I^{\,\uparrow}$ of set compositions above $I$ and the poset $\gamma(I)^{\,\downarrow}$ of integer compositions refining $\gamma(I)$. See Figure~\ref{010}.
\end{proposition}
\begin{proof}
Let $I=(\{a_1\},\ldots,\{a_n\})$, and let $[p,q]$ be a non-fermionic segment of $I$. We set $I_{p,q}=(\{a_p\},\ldots,\{a_q\})$ and denote by $J_{p,q}$ the unique permutation $(\{b_p\},\ldots,\{b_q\})$ on $[q-p+1]$ such that $a_i<a_j$ if and only if $b_i<b_j$. This implies that the posets $I_{p,q}^{\,\uparrow}$ and $J_{p,q}^{\,\uparrow}$ are isomorphic, and so their descent compositions coincide, that is, $\gamma(I_{p,q})=\gamma(J_{p,q})$. Since the elements of $I^{\,\uparrow}$ are obtained by merging blocks belonging to each $I_{p,q}$ independently, we obtain that $I^{\,\uparrow}$ is isomorphic to $J_{p_1,q_1}^{\,\uparrow}\times\cdots\times J_{p_r,q_r}^{\,\uparrow}$, where $[p_1,q_1],\ldots,[p_r,q_r]$ are the non-fermionic segments of $I$. On the other hand, proceeding as in the proof of Proposition~\ref{006}, we obtain that $\gamma(I)^{\,\downarrow}$ is isomorphic to $\gamma(J_{p_1,q_1})^{\,\downarrow}\times\cdots\times\gamma(J_{p_r,q_r})^{\,\downarrow}$. By Proposition~\ref{007}, the local map induced by $\alpha$ defines an isomorphism between $J_{p_i,q_i}^{\,\uparrow}$ and $\gamma(J_{p_i,q_i})^{\,\downarrow}$ for every $[p_i,q_i]$. Since the map $J\mapsto\alpha(J)$ acts independently on $I_{p_i,q_i}$ and preserves the fermionic blocks, it defines the desired poset isomorphism, and the result follows.
\end{proof}
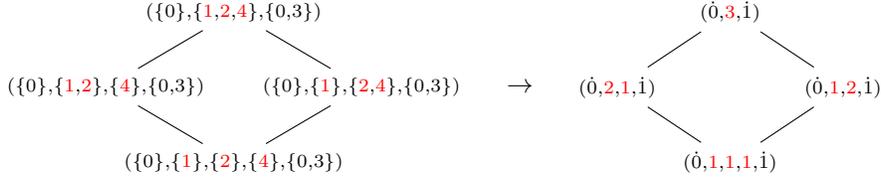
\begin{figure}[H]\centering
\begin{tikzpicture}[node distance=2cm]
\node(a)at(-0.6,2){$_{(\{0\},\{{\red1},{\red2},{\red4}\},\{0,3\})}$};
\node(b)at(-2.3,1){$_{(\{0\},\{{\red1},{\red2}\},\{{\red4}\},\{0,3\})}$};
\node(c)at(+1.1,1){$_{(\{0\},\{{\red1}\},\{{\red2},{\red4}\},\{0,3\})}$};
\node(d)at(-0.6,0){$_{(\{0\},\{{\red1}\},\{{\red2}\},\{{\red4}\},\{0,3\})}$};
\node(e)at(3.2,1){$\to$};
\node(f)at(6,2){$_{(\dot{0},{\red3},\dot{1})}$};
\node(g)at(+4.5,1){$_{(\dot{0},{\red2},{\red1},\dot{1})}$};
\node(h)at(+7.5,1){$_{(\dot{0},{\red1},{\red2},\dot{1})}$};
\node(i)at(6,0){$_{(\dot{0},{\red1},{\red1},{\red1},\dot{1})}$};
\draw(a)--(b);\draw(a)--(c);\draw(b)--(d);\draw(c)--(d);
\draw(f)--(g);\draw(f)--(h);\draw(g)--(i);\draw(h)--(i);
\end{tikzpicture}
\caption{Isomorphism between $I^{\,\uparrow}$ and $\gamma(I)^{\,\downarrow}$ for $I=(\{0\},\{1\},\{2\},\{4\},\{0,3\})$.}\label{010}
\end{figure}

\subsection{Product and coproduct rules}\label{040}

We now introduce the $Q$-basis, which is indexed by set supercompositions and is constructed using the partial order established in Subsection~\ref{005}. The \emph{$Q$-function} associated with a set supercomposition $I$ is defined by\begin{equation}\label{012}Q_I=\sum_{J\succeq I}M_J.\end{equation}For instance, if $I=(\{0,4,5\},\{1\},\{3\},\{2\},\{0\},\{7\},\{5\},\{6\})$, we have\[\begin{array}{rcl}
Q_I&=&M_{(\{0,4,5\},\{1\},\{3\},\{2\},\{0\},\{7\},\{5\},\{6\})}+M_{(\{0,4,5\},\{1,3\},\{2\},\{0\},\{7\},\{5\},\{6\})}\\[0.1cm]&&\quad+\,\,M_{(\{0,4,5\},\{1\},\{3\},\{2\},\{0\},\{7\},\{5,6\})}+M_{(\{0,4,5\},\{1,3\},\{2\},\{0\},\{7\},\{5,6\})}.\end{array}\]Applying the M\"{o}bius inversion formula \cite[Section~3.7]{St97} to the Boolean intervals in the poset of set supercompositions, we obtain\[M_I=\sum_{J\succeq I} (-1)^{\ell(I)-\ell(J)}Q_J,\]and thus the family $\{Q_I\mid I\text{ is a set supercomposition}\}$ forms a linear basis for $\sNCQSym$.

The product of the functions $Q_I$ can be described by an extension of the usual shuffle, which we call super-shuffles. A set supercomposition $K$ is called an \emph{$(I,J)$-super-shuffle} if it can be obtained from an $(I,J)$-shuffle $L$ by a possibly empty sequence of disjoint mergers of consecutive blocks $L_p,\ldots,L_q$ of $L$ into their union $L_p \cup \cdots \cup L_q$, according to the following rules:
\begin{enumerate}
\item $L_p$ is a fermionic block coming from $I$, and $(L_{p+1},\ldots,L_q)$ is an increasing sequence of $J$.
\item $L_q$ is a fermionic block coming from $J$, and $(L_p,\ldots,L_{q-1})$ is an increasing sequence of $I$.
\end{enumerate}
The set of all $(I,J)$-super-shuffles is denoted by $\SSh(I,J)$. As with quasi-shuffles, the \emph{sign} of an $(I,J)$-super-shuffle $K$ is defined as $\sgn(K)=(-1)^{\varepsilon(L)}$, where $L$ is the unique $(I,J)$-shuffle from which $K$ is obtained. For instance, if $I=(\{2,4\},\{0,1\},\{3\},\{5\})$ and $J=(\{{\red1}\},\{{\red0}\})$, the following supercompositions are $(I,J)$-super-shuffles of sign $1$, where the first sequence is the unique $(I,J)$-shuffle from which the others are obtained:
\[\begin{array}{ccc}
(\{2,4\},\{0,1\},\{{\red6}\},\{3\},\{5\},\{{\red0}\}) & (\{2,4\},\{0,1,{\red6}\},\{3\},\{5\},\{{\red0}\}) & (\{2,4\},\{0,1\},\{{\red6}\},\{3\},\{{\red0},5\}) \\[0.1cm]
(\{2,4\},\{0,1\},\{{\red6}\},\{{\red0},3,5\}) & (\{2,4\},\{0,1,{\red6}\},\{3\},\{{\red0},5\}) & (\{2,4\},\{0,1,{\red6}\},\{{\red0},3,5\})
\end{array}\]
Note that if $K$ is an $(I,J)$-super-shuffle, then each non-fermionic block of $K$ coincides with a non-fermionic block of $I$ or $J[n]$. Furthermore, each fermionic block of $K$ is either of the form $(I_p\cup\cdots\cup I_q)\cup J_b[n]$ or $I_a\cup(J_s[n]\cup\cdots\cup J_t[n])$, where $I_a$ and $J_b$ are fermionic blocks of $I$ and $J$, respectively, and $(I_p,\ldots,I_q)$ and $(J_s,\ldots,J_t)$ are possibly empty increasing sequences of $I$ and $J$.

\begin{theorem}\label{011}
Let $I$ and $J$ be two set supercompositions. Then
$$Q_I\,Q_J=\sum_{K\in\SSh(I,J)}\sgn(K)\,Q_K.$$
\end{theorem}
\begin{proof}
By the definition of the $Q$-basis in~\eqref{012} and the product rule for the $M$-basis in Proposition~\ref{001}, we can expand both sides of the desired identity in terms of the $M$-basis:\[Q_I Q_J=\Bigl(\sum_{I'\succeq I}M_{I'}\Bigr)\Bigl(\sum_{J'\succeq J}M_{J'}\Bigr)=\sum_{I'\succeq I}\sum_{J'\succeq J}\sum_{L\in\QSh(I',J')}\sgn(L)\,M_L\]and\[\sum_{K\in\SSh(I,J)}\sgn(K)\,Q_K=\sum_{K\in\SSh(I,J)}\sgn(K)\sum_{K'\succeq K}M_{K'}.\]Thus, it suffices to construct a sign-preserving bijection between the set of triples $(I',J',L)$ such that $I'\succeq I$, $J'\succeq J$, and $L\in\QSh(I',J')$, and the set of pairs $(K,K')$ such that $K\in\SSh(I,J)$ and $K'\succeq K$.

Let $K\in\SSh(I,J)$, where $I$ has bidegree $(n,m)$, and let $K'\succeq K$. We characterize the blocks of $K'$. By definition, each non-fermionic block $B$ of $K'$ is the union of an increasing sequence of blocks of $K$. Since each non-fermionic block of $K$ coincides with a non-fermionic block of either $I$ or $J$, we have that\[B=(I_p\cup\cdots\cup I_q)\cup(J_s[n]\cup\cdots\cup J_t[n]),\]where $(I_p,\ldots,I_q)$ and $(J_s,\ldots,J_t)$ are possibly empty increasing sequences of consecutive blocks of $I$ and $J$, respectively. On the other hand, the fermionic blocks of $K'$ are exactly those of $K$. Hence, each fermionic block $B$ of $K'$ is of one of the forms\[B=(I_p\cup\cdots\cup I_q)\cup J_b[n]\quad\text{or}\quad B=I_a\cup(J_s[n]\cup\cdots\cup J_t[n]),\]where $I_a$ and $J_b$ are fermionic blocks of $I$ and $J$, respectively, and the sequences $(I_p,\ldots,I_q)$, $(J_s,\ldots,J_t)$ are as above. Thus, every block $B$ of $K'$ can be written as $B=B_I\cup B_J[n]$, where $B_I$ consists of the elements coming from blocks of $I$, and $B_J$ consists of the elements coming from blocks of $J$. Now, let $I'$ be the set supercomposition formed by extracting the nonempty sets $B_I$ from the blocks of $K'$, preserving their relative order. Similarly, let $J'$ be formed by extracting the nonempty sets $B_J$ and shifting their elements back by $n$, that is, $B_J[-n]$, preserving their relative order. By the definition of the partial order, it follows that $I'\succeq I$ and $J'\succeq J$. Taking $L=K'$, we obtain $L\in\QSh(I',J')$, which yields the desired triple $(I',J',L)$.

Conversely, given $I'\succeq I$, $J'\succeq J$, and $L\in\QSh(I',J')$, each block $B$ of $L$ is uniquely of one of the following types:\[\text{(1) }\,B=(I_p\cup\cdots\cup I_q)\cup(J_s[n]\cup\cdots\cup J_t[n])\text{ (2) }\,B=(I_p\cup\cdots\cup I_q)\cup J_b[n] \text{ (3) }\,B=I_a\cup(J_s[n]\cup\cdots\cup J_t[n]),\]where $(I_p,\ldots,I_q)$ and $(J_s,\ldots,J_t)$ are possibly empty increasing sequences of $I$ and $J$, and $I_a,J_b$ are fermionic blocks of $I$ and $J$, respectively. Blocks of type~(1) are precisely the non-fermionic blocks of $L$. Splitting each block of type~(1) into the corresponding blocks of $I$ and $J$, and preserving the remaining blocks, yields $K\in\SSh(I,J)$; the relative order is preserved and $I$-blocks are placed before $J$-blocks within each split block. By construction, $K':=L\succeq K$. This recovers $(K,K')$ and completes the proof.
\end{proof}

Note that if $I$ and $J$ have only fermionic blocks, then no mergers are allowed. In this case, the product formula in the $Q$-basis reduces to the \emph{signed shuffle product}, where the sign is determined by the relative order of the fermionic blocks. For instance, if $I=(\{0,2\},\{0,1,3\})$ and $J=(\{{\red0},{\red1},{\red2}\},\{0\})$, we have\[\begin{array}{rcl}
Q_I\,Q_J &=& Q_{(\{0,2\},\{0,1,3\},\{{\red0},{\red4},{\red5}\},\{{\red0}\})}-Q_{(\{0,2\},\{{\red0},{\red4},{\red5}\},\{0,1,3\},\{{\red0}\})}+Q_{(\{0,2\},\{{\red0},{\red4},{\red5}\},\{{\red0}\},\{0,1,3\})}\\[0.1cm]&&\quad+\,Q_{(\{{\red0},{\red4},{\red5}\},\{0,2\},\{0,1,3\},\{{\red0}\})}-Q_{(\{{\red0},{\red4},{\red5}\},\{0,2\},\{{\red0}\},\{0,1,3\})}+Q_{(\{{\red0},{\red4},{\red5}\},\{{\red0}\},\{0,2\},\{0,1,3\})}.
\end{array}\]

We now describe the coproduct structure of $\sNCQSym$ in the $Q$-basis. As in the classical theory of quasi-symmetric functions, the coproduct is defined by deconcatenating a set supercomposition into an initial and a final segment, followed by a standardization procedure on each part.
\begin{proposition}\label{013}
Let $I=(I_1,\ldots,I_k)$ be a set supercomposition. Then the coproduct of $Q_I$ is given by\[\Delta(Q_I)=\sum_{i=0}^kQ_{\std(I_1,\ldots,I_i)}\otimes Q_{\std(I_{i+1},\ldots,I_k)}.\]
\end{proposition}
\begin{proof}
Applying~\eqref{012} and the coproduct formula in Proposition~\ref{028} to the monomial basis, we obtain\[\Delta(Q_I)=\sum_{J\succeq I}\Delta(M_J)=\sum_{J\succeq I}\sum_{s=0}^{\ell(J)}M_{\std(J_1,\dots,J_s)}\otimes M_{\std(J_{s+1},\dots,J_{\ell(J)})}.\]On the other hand, for each $i\in [k]_0$, we have\[Q_{\std(I_1,\ldots,I_i)}\otimes Q_{\std(I_{i+1},\ldots,I_k)}=\Bigl(\,\sum_{K\succeq\std(I_1,\ldots,I_i)}\!\!M_K\Bigr)\otimes\Bigl(\,\sum_{L\succeq\std(I_{i+1},\ldots,I_k)}\!\!M_L\Bigr).\]

Now, fix $i\in [k]_0$ and define $\Cut(i)=\{(J,s)\mid J\succeq I,\,s\in[\ell(J)]_0,\text{ and }J_1\cup\cdots\cup J_s=I_1\cup\cdots\cup I_i\}$. Since each $J\succeq I$ is obtained by merging adjacent blocks of $I$, it follows that for any such $J$, there is at most one index $s$ satisfying this condition. Moreover, $\Cut(i)$ consists of the pairs $(J,s)$ with $J\succeq I$ such that every block of $J$ is contained either in $I_1\cup\cdots\cup I_i$ or in $I_{i+1}\cup\cdots\cup I_k$. Thus, for $(J,s)\in\Cut(i)$, we have $\std(J_1,\ldots,J_s)\succeq\std(I_1,\dots,I_i)$ and $\std(J_{s+1},\ldots,J_{\ell(J)})\succeq\std(I_{i+1},\dots,I_k)$.

Conversely, given $K\succeq\std(I_1,\ldots,I_i)$ and $L\succeq\std(I_{i+1},\ldots,I_k)$, let $\tilde{K}$ and $\tilde{L}$ be the unique set supercompositions obtained by relabeling the elements of $K$ and $L$ via the unique order-preserving bijections onto $I_1\cup\cdots\cup I_i$ and $I_{i+1}\cup\cdots\cup I_k$, respectively. The concatenation $J:=(\tilde{K},\tilde{L})$ then defines an element $J\succeq I$ such that $(J,\ell(K))\in\Cut(i)$.

This construction establishes a bijection between $\Cut(i)$ and the pairs $(K,L)$ as described above. Therefore, the contribution to $\Delta(Q_I)$ coming from $\Cut(i)$ is\[\sum_{K\succeq\std(I_1,\ldots,I_i)} \sum_{L\succeq\std(I_{i+1},\ldots,I_k)}M_K\otimes M_L = Q_{\std(I_1,\ldots,I_i)}\otimes Q_{\std(I_{i+1},\ldots,I_k)}.\]Since each pair $(J,s)$ in the expansion of $\Delta(Q_I)$ belongs to exactly one $\Cut(i)$, summing these contributions over all $i\in[k]_0$ yields the desired result.
\end{proof}

\subsection{Hopf superalgebra of superpermutations}\label{041}

Here we give a superspace analogue of the algebra of Malvenuto--Reutenauer as the sub-Hopf superalgebra indexed by superpermutations.

Let $\sFQSym$ be the subspace of $\sNCQSym$ spanned by the family of functions $Q_I$ with $I$ a superpermutation. Elements in $\sFQSym$ will be called \emph{free quasisymmetric functions in superspace}.

\begin{proposition}\label{033}
The space $\sFQSym$ is a Hopf subalgebra of $\sNCQSym$.
\end{proposition}
\begin{proof}
By Theorem~\ref{011}, the product $Q_IQ_J$ expands in the $Q$-basis as a sum over $(I,J)$-super-shuffles. By the definition of super-shuffles, if $I$ and $J$ are superpermutations, then every $(I,J)$-super-shuffle is also a superpermutation. Consequently, $\sFQSym$ is closed under the product and forms a subalgebra. Furthermore, by Proposition~\ref{013}, the coproduct $\Delta(Q_I)$ is defined by deconcatenation and standardization. Since the standardization of any subsequence of a superpermutation is clearly a superpermutation, $\sFQSym$ is closed under the coproduct. Thus, $\sFQSym$ is a sub-Hopf superalgebra of $\sNCQSym$.
\end{proof}
We call $\sFQSym$ the \emph{Hopf superalgebra of free quasisymmetric functions in superspace}.

\medskip

Following the construction for the classical Malvenuto--Reutenauer algebra in \cite[Subsection~1.3]{AgSo02}, now we define a monomial basis for the algebra $\sFQSym$. To this goal we first need to introduce the inversion set of a superpermutation and a superspace version of the left weak order.

Given a superpermutation $I$, we define its \emph{associated permutation} $w(I)\in\sym_\infty$ as the word obtained by concatenating the nonzero elements of its blocks from left to right, where the elements within each block are arranged in strictly increasing order. For instance, if $I=(\{3\},\{0,{\blue1},{\blue5}\},\{4\},\{0\},\{2\})$, we have $w(I)=3{\blue15}42$.

The \emph{inversion set} of a superpermutation $I$ is defined as the inversion set of its associated permutation, that is, $\inv(I)=\inv(w(I))$. Given two superpermutations $I$ and $J$, we say that $I$ precedes $J$ in the \emph{super left weak order}, denoted by $I\preceq_WJ$, if $\alpha(I)=\alpha(J)$ and $\inv(I)\subseteq\inv(J)$. This relation defines a partial order on the set of all superpermutations $\sym$. Notice that the condition $\alpha(I)=\alpha(J)$ ensures that comparable superpermutations have the exact same block structure and fermionic components.

Before introducing the monomial basis for $\sFQSym$, we prove the following proposition, which ensures that the poset defined by $\preceq_W$ is isomorphic to a disjoint union of intervals of the classical left weak order on symmetric groups. Consequently, the M\"{o}bius function $\mu_W$ naturally coincides with the classical one, that is, $\mu_W(I,J)=\mu(w(I),w(J))$ if $I\preceq_W J$, and $0$ otherwise.
\begin{proposition}\label{019}
For each $\sigma\in\alpha(\sym)$, we have $\alpha^{-1}(\{\sigma\})=[I, J]$, where
\begin{enumerate}
\item $I=(I_1,\ldots,I_k)$ is the unique superpermutation such that $\alpha(I)=\sigma$ and, for each $i\in[k-1]$, every nonzero element in $I_i$ is strictly smaller than every nonzero element in $I_{i+1}$.
\item $J=(J_1,\ldots,J_k)$ is the unique superpermutation such that $\alpha(J)=\sigma$ and, for each $i\in[k-1]$, every nonzero element in $J_i$ is strictly greater than every nonzero element in $J_{i+1}$.
\end{enumerate}
Furthermore, the poset $[I,J]$ is isomorphic to the interval $[w(I),w(J)]$ in the classical left weak order via the map $w$. See Figure~\ref{018}.
\end{proposition}
\begin{proof}
Since $I\in\alpha^{-1}(\{\sigma\})$ and $\inv(I)=\emptyset$, it follows that $I\preceq_WK$ for all $K\in\alpha^{-1}(\{\sigma\})$. Let $n$ be the length of $w(I)$, and let $P$ be the set of all position pairs $(i,j)$ with $1\leq i<j\leq n$ that correspond to elements in different blocks of $I$. Because the map $w$ forces elements within the same block to be sorted in strictly increasing order, any $K\in\alpha^{-1}(\{\sigma\})$ must satisfy $\inv(K)\subseteq P$. By construction, $J\in\alpha^{-1}(\{\sigma\})$ and $\inv(J)=P$, which yields $K\preceq_WJ$ for all $K\in\alpha^{-1}(\{\sigma\})$. Therefore, $\alpha^{-1}(\{\sigma\})=[I,J]$.

The restriction of $w$ to $[I,J]$ is clearly an order-preserving injection, and by the definition of the super left weak order, it maps into $[w(I),w(J)]$. It suffices to show that this restricted map is surjective. Let $\tau\in[w(I),w(J)]$. Then $\inv(\tau)\subseteq\inv(w(J))=\inv(J)=P$. Because $\tau$ lacks inversions among positions corresponding to the same block, the superpermutation $K:=(\tau(I_1),\ldots,\tau(I_k))$ correctly preserves the internal increasing order of each block. Thus, $K\in\alpha^{-1}(\{\sigma\})$ and $w(K)=\tau$. Therefore, $[I,J]$ and $[w(I),w(J)]$ are isomorphic.
\end{proof}
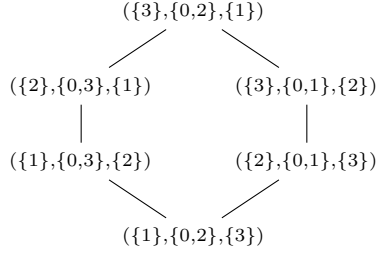
\begin{figure}[H]\centering
\begin{tikzpicture}[node distance=2cm]
\node(a)at(+0,0){$_{(\{1\},\{0,2\},\{3\})}$};
\node(b)at(-1.5,1){$_{(\{1\},\{0,3\},\{2\})}$};
\node(c)at(+1.5,1){$_{(\{2\},\{0,1\},\{3\})}$};
\node(d)at(-1.5,2){$_{(\{2\},\{0,3\},\{1\})}$};
\node(e)at(+1.5,2){$_{(\{3\},\{0,1\},\{2\})}$};
\node(f)at(+0,3){$_{(\{3\},\{0,2\},\{1\})}$};
\draw(a)--(b)--(d)--(f);\draw(a)--(c)--(e)--(f);
\end{tikzpicture}
\caption{Interval of superpermutations $I\in\sym$ with $\alpha(I)=(1,\dot{1},1)$}\label{018}
\end{figure}

\begin{remark}
The super left weak order admits an equivalent characterization in terms of the length function and the left action of the symmetric group, mirroring the classical definition. For a superpermutation $I$, we define its \emph{length} as $\len(I)=|\inv(I)|$. Note that $\len(I)=\len(w(I))$. Furthermore, for every $\sigma\in\sym_\infty$, we denote by $\sigma(I)$ the superpermutation obtained by replacing each nonzero element $x$ in the blocks of $I$ with $\sigma(x)$. Now, by definition, for any two superpermutations $I$ and $J$, we have $I\preceq_WJ$ if and only if $\alpha(I)=\alpha(J)$ and there exists a permutation $\sigma\in\sym_\infty$ such that $J=\sigma(I)$ and $\len(J)=\len(\sigma)+\len(I)$. Notice that the length condition ensures that $\sigma$ strictly adds inversions and never inverts the relative order of elements that form non-inversions in $I$.
\end{remark}

We can now introduce the monomial basis for $\sFQSym$. For each superpermutation $I$, we define the \emph{monomial free quasisymmetric function in superspace} as\[\M_I=\sum_{J\succeq_WI}\mu_W(I,J)Q_J.\]Applying the M\"{o}bius inversion formula \cite[Section~3.7]{St97} to the intervals in the poset of superpermutations, we obtain\[Q_I=\sum_{J\succeq_WI}\M_J,\]and thus the family $\{\M_I\mid I\text{ is a superpermutation}\}$ forms a linear basis for $\sFQSym$. For instance, if $I=(\{1\},\{0,2\},\{3\})$, we have\[\M_I=Q_{(\{1\},\{0,2\},\{3\})}-Q_{(\{2\},\{0,1\},\{3\})}-Q_{(\{1\},\{0,3\},\{2\})}+Q_{(\{3\},\{0,2\},\{1\})}.\]

We first study the product of monomial free quasisymmetric functions.
\begin{proposition}
Let $I$ and $J$ be two superpermutations. The product of their corresponding monomial free quasisymmetric functions is given by\[\M_I\M_J=\sum_{L\in\sym}c_{I,J}^L\,\M_L,\]where the structure coefficients $c_{I,J}^L$ are integers computed by the convolution\[c_{I,J}^L=\sum_{I'\succeq_WI}\sum_{J'\succeq_WJ}\sum_{\substack{K\in\SSh(I',J')\\K\preceq_WL}}\sgn(K)\mu_W(I,I')\mu_W(J,J').\]
\end{proposition}
\begin{proof}
By definition and applying Theorem~\ref{011}, we obtain\[\begin{array}{rcl}
\M_I\M_J&=&{\displaystyle\Bigl(\sum_{I'\succeq_WI}\mu_W(I,I')Q_{I'}\Bigr)\Bigl(\sum_{J'\succeq_WJ}\mu_W(J,J')Q_{J'}\Bigr)}\\[0.5cm]
&=&{\displaystyle\sum_{I'\succeq_WI}\sum_{J'\succeq_WJ}\mu_W(I,I')\mu_W(J,J')\sum_{K\in\SSh(I',J')}\sgn(K)Q_K.}\\[0.5cm]
&=&{\displaystyle\sum_{I'\succeq_WI}\sum_{J'\succeq_WJ}\mu_W(I,I')\mu_W(J,J')\sum_{K\in\SSh(I',J')}\sgn(K)\sum_{L\succeq_WK}\M_L.}
\end{array}\]Finally, interchanging the order of summation to extract the coefficient of a fixed superpermutation $L$, we group all terms where $K\preceq_W L$. This directly yields the structure coefficients $c_{I,J}^L$ as stated.
\end{proof}

To describe the coproduct of the monomial basis, we first need to introduce some concepts. For a superpermutation $I=(I_1,\ldots,I_k)$, an index $d\in[k-1]$ is called a \emph{global descent} of $I$ if for every nonzero element $a\in I_i$ with $i\leq d$, and every nonzero element $b\in I_j$ with $j>d$, we have $a>b$. We denote the set of global descents of $I$ by $\GDes(I)$. By convention, we also include $0$ and $k$ in $\GDes(I)$. Recall that the standardization of a word $u$ is the unique permutation $\std(u)$ having the same inversions as $u$. 

\begin{proposition}\label{020}
For any superpermutation $I=(I_1,\ldots,I_k)$, the coproduct of the function $\M_I$ is given by deconcatenating $I$ exclusively at its global descents:\[\Delta(\M_I)=\sum_{i\in\GDes(I)}\M_{\std(I_1,\ldots,I_i)}\otimes\M_{\std(I_{i+1},\ldots,I_k)}.\]
\end{proposition}
\begin{proof}
By definition and the coproduct formula for the $Q$-basis in Proposition~\ref{013}, we have\[\Delta(\M_I)=\sum_{J\succeq_W I}\mu_W(I,J)\Delta(Q_J)=\sum_{J\succeq_W I}\mu_W(I,J)\Bigl(\,\sum_{j=0}^{k}Q_{\std(J_1,\ldots,J_j)}\otimes Q_{\std(J_{j+1},\ldots,J_k)}\Bigr).\]Recall that the condition $J\succeq_W I$ forces $\alpha(J)=\alpha(I)$, ensuring that $J$ has the exact same number of blocks and fermionic structure as $I$. Interchanging the order of summation, we can group the terms based on the resulting standardizations of the left and right factors:\[\Delta(\M_I)=\sum_{j=0}^k\sum_{A,B}\Bigl(\,\sum_{J\in S_{A,B}}\mu_W(I,J)\Bigr)Q_A\otimes Q_B,\]where $A\succeq_W\std(I_1,\ldots,I_j)$ and $B\succeq_W\std(I_{j+1},\ldots,I_k)$, and the inner sum runs over the fiber of superpermutations $S_{A,B}=\{J\succeq_WI\mid\std(J_1,\ldots,J_j)=A\text{ and }\std(J_{j+1},\ldots,J_k)=B\}$.

For each split point $j$, we analyze the inner sum. Suppose $j\notin\GDes(I)$. Let $p$ be the total number of nonzero elements in the first $j$ blocks of $I$. By Proposition~\ref{020}, the map $w$ is an order-preserving isomorphism from the super left weak order on the set $\alpha^{-1}(\{\alpha(I)\})$ to a classical interval of the left weak order. The condition that $J$ restricts to the standardizations $A$ and $B$ is entirely determined by the relative order of its elements, meaning it translates directly to the classical standardizations of the word $w(J)$ cut at position $p\leq n$, where $n$ is the number of nonzero elements in the blocks of $I$. That is, $w(\std(J_1,\ldots,J_j))=\std(w(J)|_{[1,p]})$, where $w(J)|_{[1,p]}$ corresponds to the first $p$ entries of the permutation $w(J)\in\sym_n$. Thus,
\[w(S_{A,B})=\{u_1\cdots u_n\succeq_Lw(I)\mid\std(u_1\cdots u_p)=w(A)\text{ and }\std(u_{p+1}\cdots u_n)=w(B)\}=:S_{w(I)}.\]By Proposition~\ref{019}, we obtain\[\sum_{J\in S_{A,B}}\mu_W(I,J)=\sum_{u\in S_{w(I)}}\mu(w(I),u).\]By definition, $j\notin\GDes(I)$ implies that $p\notin\GDes(w(I))$. By \cite[Theorem~3.1]{AgSo02}, the sum above over this specific fiber evaluates to zero when the cut is not a global descent. Consequently, all terms for $j\notin\GDes(I)$ vanish.

Finally, if $j\in\GDes(I)$, every nonzero element in the first $j$ blocks is strictly greater than every nonzero element in the remaining blocks. In this case, no superpermutation $J\succeq_W I$ can add inversions across the cut $j$. The interval $[I,J]$ factors as $[\std(I_1,\ldots,I_j),A]\times[\std(I_{j+1},\ldots,I_k),B]$, yielding $\mu_W(I,J)=\mu_W(\std(I_1,\ldots,I_j),A)\,\mu_W(\std(I_{j+1},\ldots,I_k),B)$. Substituting this factorization back into the sum separates the $A$ and $B$ components entirely, that is,\[\Delta(\M_I)=\sum_{j\in\GDes(I)}\underbrace{\Bigl(\sum_A\mu_W(\std(I_1,\ldots,I_j),A)Q_A\Bigr)}_{\M_{\std(I_1,\ldots,I_j)}}\otimes\underbrace{\Bigl(\sum_B\mu_W(\std(I_{j+1},\ldots,I_k),B)Q_B\Bigr)}_{\M_{\std(I_{j+1},\ldots,I_k)}}.\]This concludes the proof.
\end{proof}
For instance, if $I=(\{0,6\},\{3\},\{0,4,5\},\{1\},\{2\})$, we have $\GDes(I)=\{0,1,3,5\}$, and so\[\Delta(\M_I)=1\otimes\M_I+\M_{(\{0,1\})}\otimes\M_{(\{3\},\{0,4,5\},\{1\},\{2\})}+\M_{(\{0,4\},\{1\},\{0,2,3\})}\otimes\M_{(\{1\},\{2\})}+\M_I\otimes1.\]

\subsection{An application to fundamental quasisymmetric functions in superspace}\label{025}

There is a natural algebra homomorphism $\Q^\theta\dangle{x}\rightarrow\Q^\theta\dbrack{x}$, called the \emph{abelianization morphism}, which sends each noncommutative variable $x_i$ to its commutative counterpart and acts as the identity on the variables $\theta_i$.

We now describe the restriction of the abelianization morphism to $\sNCQSym$.
\begin{proposition}\label{014}
The abelianization morphism restricted to $\sNCQSym$ induces an algebra homomorphism $\pi:\sNCQSym\rightarrow\sQSym$. Moreover, for any set supercomposition $I$, we have $\pi(M_I)= M_{\alpha(I)}$.
\end{proposition}
\begin{proof}
Since $\pi$ is the restriction of the abelianization morphism, it is immediately an algebra homomorphism. Hence, it suffices to show that $\pi(M_I)=M_{\alpha(I)}$ for all set supercomposition $I=(I_1,\ldots,I_k)$. Consider a monomial $u=\theta_{i_1}\cdots \theta_{i_m}x_{j_1}\cdots x_{j_n}$ occurring in $M_I$, and let $\pi(u)$ be the image of $u$ under the abelianization morphism. Then $\pi(u)=\theta_{i_1}\cdots \theta_{i_m}x_{a_1}^{b_1}\cdots x_{a_k}^{b_k}$, where $a_1<\cdots<a_k$ are the indices of $u$, $b_t=|I_t|$ if $I_t$ is non-fermionic, and $b_t=|I_t|-1$ if $I_t$ is fermionic. Hence $\pi(u)\in M_{\alpha(I)}$. Conversely, given a monomial $w \in M_{\alpha(I)}$, one reconstructs a monomial $u \in M_I$ mapping to $w$ by reversing the above construction.
\end{proof}

\begin{proposition}\label{015}
For any superpermutation $I$, we have $\pi(Q_I)=L_{\gamma(I)}$.  
\end{proposition}
\begin{proof}
This is a direct consequence of Proposition~\ref{009} and Proposition~\ref{014}.
\end{proof}

The following result is analogous to \cite[Proposition~4.6]{FiGaLaPi25}.
\begin{theorem}\label{002}
Let $\alpha$ and $\beta$ be dotted compositions. Then, for any superpermutations $I$ and $J$ satisfying $\gamma(I)=\alpha$ and $\gamma(J)=\beta$, we have\[L_\alpha\,L_\beta=\sum_{K\in\SSh(I,J)}\sgn(K)\,L_{\gamma(K)}.\]Moreover, the right-hand side depends only on $\alpha$ and $\beta$.
\end{theorem}
\begin{proof}
Since $\pi$ is an algebra homomorphism, Proposition~\ref{015} yields $\pi(Q_IQ_J)=\pi(Q_I)\pi(Q_J)=L_{\alpha}L_{\beta}$. On the other hand, using the product formula in Theorem~\ref{011}, and applying again Proposition~\ref{015}, we obtain\[\pi(Q_IQ_J)=\pi\Bigl(\sum_{K\in\SSh(I,J)}\sgn(K)\,Q_K\Bigr)=\sum_{K\in\SSh(I,J)}\sgn(K)\,\pi(Q_K)=\sum_{K\in\SSh(I,J)}\sgn(K)\,L_{\gamma(K)}.\]
Comparing both expressions concludes the proof.
\end{proof}
Note that when no dotted parts occur, the formula in Theorem~\ref{002} reduces to the classical product of fundamental quasisymmetric functions~\cite[Equation (3.13)]{LuMyWi13}.

For instance, if $\alpha=\gamma(I)=(2,\dot{1})$ and $\beta=\gamma(J)=(2)$, where $I=(\{1\},\{2\},\{0,3\})$ and $J=(\{1\},\{2\})$, we have\[\SSh(I,J)=\left\{\begin{array}{cccc}
_{(\{1\},\{2\},\{0,3\},\{4\},\{5\})},&_{(\{1\},\{2\},\{0,3,4\},\{5\})},&_{(\{1\},\{2\},\{0,3,4,5\})},&_{(\{1\},\{2\},\{4\},\{0,3\},\{5\})},\\
_{(\{1\},\{2\},\{4\},\{0,3,5\})},&_{(\{1\},\{4\},\{2\},\{0,3\},\{5\})},&_{(\{1\},\{4\},\{2\},\{0,3,5\})},&_{(\{4\},\{1\},\{2\},\{0,3\},\{5\})},\\
_{(\{4\},\{1\},\{2\},\{0,3,5\})},&_{(\{1\},\{2\},\{4\},\{5\},\{0,3\})},&_{(\{1\},\{4\},\{2\},\{5\},\{0,3\})},&_{(\{1\},\{4\},\{5\},\{2\},\{0,3\})},\\
_{(\{4\},\{1\},\{2\},\{5\},\{0,3\})},&_{(\{4\},\{1\},\{5\},\{2\},\{0,3\})},&_{(\{4\},\{5\},\{1\},\{2\},\{0,3\})}
\end{array}\right\},\]and so\[\begin{array}{rcl}L_\alpha L_\beta&=&L_{(2,\dot{1},2)}+L_{(2,\dot{2},1)}+L_{(2,\dot{3})}+L_{(3,\dot{1},1)}+L_{(3,\dot{2})}+L_{(2,1,\dot{1},1)}+L_{(2,1,\dot{2})}+L_{(1,2,\dot{1},1)}+L_{(1,2,\dot{2})}\\[0.15cm]&&\,+L_{(4,\dot{1})}+L_{(2,2,\dot{1})}+L_{(3,1,\dot{1})}+L_{(1,3,\dot{1})}+ L_{(1,2,1,\dot{1})}+L_{(2,2,\dot{1})}.\end{array}\]

\section*{Concluding remarks}

The theory of combinatorial Hopf superalgebras in superspace remains an active area of development. The construction of the algebra of free quasisymmetric functions in superspace suggests that other fundamental combinatorial Hopf algebras, such as the Loday--Ronco and Solomon descent algebras, may admit natural extensions to this setting. Exploring these extensions could provide further insight into the combinatorics of superspace and potentially clarify structural connections with Lie superalgebras.

We also emphasize that our construction relies on one of the two partial orders on dotted compositions introduced in~\cite[Definition~5.12, Equation~(5.18)]{FiLaPi19}. The alternative order, which allows for more intricate interactions between dotted and non-dotted components, remains less understood and represents a promising avenue for future exploration.

In the context of symmetric functions, fundamental quasisymmetric functions in superspace can be viewed as a refinement of Schur-type structures. However, a satisfactory definition of Schur quasisymmetric functions in superspace and Schur functions in $\sNCSym$, for instance, via a Jacobi--Trudi-type determinant, remains an open problem.

Quasisymmetric functions are closely linked to descent-type statistics and the combinatorics of permutations, including structures arising from the weak and Bruhat orders. In our setting, the partial order introduced in Subsection~\ref{005} depends solely on the non-fermionic components of a superpermutation. Notably, fermionic blocks do not affect the descent structure encoded by this order. This suggests that fermionic components offer a means to refine combinatorial structures while preserving descent-related information, potentially providing a robust framework for studying permutation statistics compatible with such orders.

\subsection*{Acknowledgments}

The first named author acknowledges the financial support of DIDULS/ULS, through
the project PR2553853. The second named author was partially supported by the grant ANID-FONDECYT Iniciaci\'on No. 11241418. The third named author acknowledges the financial support of Fondo de Apoyo a la Investigaci\'on DIUA309-2025.

\bibliographystyle{plainurl}\bibliography{bibtex.bib}
\end{document}